\documentclass{amsproc}%
\usepackage{amsmath}
\usepackage{amssymb}
\usepackage{amsthm}

\newcommand{\BR}{\mathbb{{R}}}
\newcommand{\T}{\mathbb{{T}}}
\newcommand{\N}{\mathbb{{N}}}

\RequirePackage{srcltx}

\def\le {\leqslant}
\def\ge {\geqslant}

\def\esssup{\text{ess sup}}

\usepackage{amssymb}
\usepackage{amsfonts}
\usepackage{amsmath}
\usepackage{graphicx}%

\providecommand{\U}[1]{\protect\rule{.1in}{.1in}}
\theoremstyle{plain}

\newtheorem{corollary}{Corollary}[section]

\newtheorem*{definition}{Definition}

\newtheorem{lemma}{Lemma}[section]

\newtheorem{theorem}{Theorem}[section]

\numberwithin{equation}{section}

\theoremstyle{definition}
\newtheorem{remark}{Remark}[section]
\newtheorem*{examples}{Examples}
\newtheorem*{example}{Example}

\begin{document}

\title[Bernstein inequalities with nondoubling weights]
{Bernstein inequalities with nondoubling weights
}
\author{A. Bondarenko}
\address{A. Bondarenko\\
Department of Mathematical Analysis, National Taras Shevchenko
University, str.\ Volodymyrska, 64, Kyiv, 01033, Ukraine and
Department of Mathematical Sciences, Norwegian University of Science and Technology, NO-7491 Trondheim, Norway}
\email{andriybond@gmail.com}

\author{S. Tikhonov}
\address{S. Tikhonov
\\
ICREA
 and Centre de Recerca Matem\`{a}tica
\\ Apartat 50 08193 Bellaterra, Barcelona
}
\email{stikhonov@crm.cat}

\thanks{This paper is in final form and no version of it will be submitted for
publication elsewhere.}
\thanks{
This research was partially supported by the MTM2011-27637/MTM,
2009 SGR 1303, RFFI 12-01-00169, and NSH 979.2012.1.
This work was carried out during the tenure of an ERCIM ``Alain Bensoussan'' Fellowship Programme. The research leading to these results has received funding from the European Union Seventh Framework Programme (FP7/2007-2013) under grant agreement N 246016.
}

\date{August 24, 2012}

\begin{abstract}
We answer Totik's question on weighted Bernstein's inequalities 
 showing that
\begin{equation*}
\|T_n'\|_{L_p(\omega)} \le C(p,\omega)\, {n}\,\|T_n\|_{L_p(\omega)},\qquad 0<p\le \infty,
\end{equation*}
holds for all trigonometric polynomials $T_n$ and certain  nondoubling weights $\omega$. Moreover, we find necessary conditions on $\omega$ for
Bernstein's inequality to hold. We also prove weighted
Bernstein-Markov, Remez, and Nikolskii inequalities for trigonometric and algebraic polynomials.
\end{abstract} \maketitle

\vspace{5mm}

{{
\tableofcontents
}}
\vspace{4mm}
\section{Introduction}
The famous  Bernstein inequality for trigonometric polynomials $T_n$ of degree at most $n$
\begin{equation} \label{ber}
\|T_n'\|_{L_p(\T)} \le C {n}\|T_n\|_{L_p(\T)}
\end{equation}
plays an important role in the modern analysis. Here,
$\|\cdot\|_{L_p(\T)}$ is the  $L_p$-(quasi)norm, i.e.,
$$\|f\|_{L_p(\T)}=\left(\int_{\mathbb{T}}|f(t)|^p \,dt\right)^{1/p}, \qquad 0< p < \infty,$$
with the usual modification for $p=\infty$. Bernstein proved
(\ref{ber}) for $p=\infty$; the case $p<\infty$ was done by Zygmund
\cite{zy}. The best constant  $C$ is equal to $1$ for any $p\in
(0,\infty]$, see \cite{riesz, zy, ar}.

For algebraic polynomials $P_n$ of degree at most $n$, the Bernstein inequality is given by
\begin{equation*} \label{ber1}
|P_n'(x)| \le \frac{n}{\sqrt{1-x^2}}\|P_n\|_{C[-1,1]},\qquad\qquad
x\in (-1,1),
\end{equation*}
where $\|\cdot\|_{C[-1,1]}$ denotes the supremum norm on the $[-1,1]$.
Its $L_p$-version is written as follows:
\begin{equation} \label{ber-ber}
\|{\sqrt{1-x^2}} P_n'(x)\|_{L_p{[-1,1]}} \le
C(p){n}\|P_n\|_{L_p{[-1,1]}}, \qquad 0< p\le \infty.
\end{equation}
Another important inequality for the derivative of algebraic
polynomials is the following Markov inequality:
\begin{equation} \label{ber-mar}
\|P_n'\|_{L_p[-1,1]} \le C(p){n^2}\|P_n\|_{L_p[-1,1]},\qquad 0 <p\le
\infty.
\end{equation}

Both Bernstein and Bernstein--Markov inequalities for trigonometric and algebraic polynomials respectively
 were extended to the case of smaller intervals  ( Privalov, Jackson, and Bary; see, e.g., \cite{bar})
 and several intervals (see the recent paper by Totik \cite{to}).

In this paper we study weighted analogues of Bernstein's inequality
\begin{equation} \label{berw}
\|T_n'\|_{L_p(\omega)} \le C(p,\omega) {n}\|T_n\|_{L_p(\omega)},
\end{equation}
where 
  $\omega$ is a weight function, i.e., a nonnegative integrable function on $\T$.
Here and in what follows,
$\|T_n\|_{L_p(\omega)}=
\left(\int_{\mathbb{T}}|T_n|^p\omega\right)^{1/p}$ if $p<\infty$ and
$\|T_n\|_{L_\infty(\omega)}=\esssup_{t\in\mathbb{T}}|T_n(t)\omega(t)|$.

First, we note that Muckenhoupt's $A_p$ condition on weights ensures that  (\ref{berw}) holds for $1<p<\infty$. This follows from the fact that the Marcinkiewicz multiplier theorem
and Littlewood-Paley decomposition  hold in $L_p(\omega)$ with
$\omega\in A_p$.
In \cite{mas1}, Mastroianni and Totik proved a much stronger result that for any weight $\omega$ satisfying the
{\it doubling} condition and for $1\le p <\infty$ inequality (\ref{berw}) holds. Later, a similar result was shown for $0< p< 1$ (see
\cite{er}).

We recall that a periodic weight function $\omega$ satisfies the
doubling condition if
\begin{equation} \label{doubl}
 W(2I) \le L W(I)
\end{equation}
for all intervals $I$, where $L$ is a constant independent of $I$, $2I$ is the interval 
  twice the length of $I$ and with the midpoint coinciding with that of $I$, and
$$
W(I)=\int_I \omega(t)\, dt.
$$
Let us also recall that a weight $\omega$ satisfies the $A_\infty$ condition if for every $\alpha>0$ there is $\,\beta>0$ such that
\begin{equation*}
 W(E) \ge  \beta W(I)
\end{equation*}
for any interval $I$ and any measurable set $E\subset I$ with
$|E|\ge  \alpha |I|$. It is known \cite[Ch. V]{stein} that any
$A_\infty$ weight satisfies the doubling condition.
Here and in what follows, $|E|$ denotes the Lebesgue measure of the set  $E$.

For the supremum norm, in addition to the natural assumption that $\omega$ is bounded, one needs the $A^*$ condition, i.e.,
there exists a constant $L$ such that for all
intervals $I\subset [-\pi, \pi]$ and $t \in I$ we have
$$\omega(t)\le \frac L{|I|} \, W(I). 
$$
This condition is stronger than the $A_\infty$ condition and it is sufficient for (\ref{berw}) to hold when $p=\infty$.

In \cite{totik-q}, Totik posed the following  question: under which
condition on a general (not necessary doubling) weight $\omega$ does the
 Bernstein inequality (\ref{berw}) hold for any trigonometric
polynomial $T_n$ of degree at most $n$? In this paper we aim to answer this
question. We deal with the weight functions from the  class $\Omega$.

\begin{definition}
Let
$$
\omega(t)=\exp{\left(-F(g(t))\right)},\quad t\in \mathbb{T},
$$
where $g:\mathbb{T}\to [-A,A]$, $A>0$, is an 
 analytic function, e.g.,
\begin{equation}
\label{g}|g^{(n)}(t)|\le D^nn!,\quad t\in\mathbb{T},\quad
n=1,2,\ldots,
\end{equation}
 such that each zero of $g$ is of multiplicity one.
 Let also
$F:[-A,A]\setminus\left\{0\right\}\to (0,\infty)$ be an even
$C^{\infty}$ function on $(0,A]$ such that
$$
 F(x)\to\infty\;\;\;\mbox{as}\quad x\to 0+;\qquad\qquad\qquad\qquad\qquad\qquad\qquad
 \leqno{(F1)}
$$
$$
F\quad\mbox{ is decreasing on}\quad (0,A];\qquad\qquad\qquad\qquad\qquad\qquad\;
\leqno{(F2)}
$$
$$
|F^{(n)}(x)|\le B^nn^n\frac{F(x)}{x^n},\quad
x\in(0,A],\quad n=1,2,\ldots;\qquad
\leqno{(F3)}
$$
$$
\mbox{there exist} \quad A_1, A_2>0\quad\mbox{ such that}\;\;\qquad\qquad\qquad\qquad
\leqno{(F4)}
$$
$$
A_2\le\frac{|F'(x)|x}{F(x)}\le A_1,\quad x\in(0,A].\qquad\qquad\qquad\qquad\qquad
$$
Then we write that $\omega\in \Omega$.
\end{definition}
It is worth mentioning that all our results hold for weights $\omega(t)=\exp{\left(-F(g(t))\right)}$, where $F$ satisfies $(F_1)-(F_4)$ only for $x\in (0,\varepsilon)$ for some $0<\varepsilon<A$ and $$|F^{(n)}(x)|\le B^n n^n {F(x)},\quad
x\in [\varepsilon,A],\quad n=1,2,\ldots.
$$

The typical example of the function $g$ is $\sin t$ or $\cos t$.  Note that $\omega\in \Omega$ is nondoubling if and only if $g$ has at least one zero on $\T$. In what follows this  will be assumed to be the case.
Below we give some examples of a function $F$ satisfying properties
$(F1)-(F4)$. Consider a positive even function $F$ defined on $(0, A]$.

\begin{examples}
{\textnormal{
\\
1. Let $$F(x)=x^{-\alpha}, \; 
 {x^{-\alpha}} |\log x|^{\xi_1},\;  
{x^{-\alpha}}|\log x|^{\xi_1} \cdots |\log_k x|^{\xi_k},\;
{x^{-\alpha}}\exp {|\log x|^\xi},
$$
where $\alpha>0$, $\xi_j\in \BR,$ $\xi\in(0,1)$, and $\log_j x=\log_{j-1}|\log x|$.
Note that any such a function $F$ is of regular variation of index  $-\alpha$, i.e.,
 for all $r > 0$,
\begin{equation}
\label{F-r-x}
\lim_{x\to 0+}\frac{F(rx)}{F(x)}=r^{-\alpha},
\end{equation}
or, equivalently,
$$
F(x)=\frac{1}{x^\alpha}\eta(x),
$$
where $\eta$ is a slowly varying function, i.e., $\lim\limits_{x\to 0+}\frac{\eta(rx)}{\eta(x)}=1$.
 \\
2. Note that there are functions satisfying
$(F1)$--$(F4)$ which are not regularly varying. For example, the function
$$
F(x)=\exp\Big\{-\log x\big(2+\sin \big(
\log_3 x
\big)\big)\Big\}
$$
is such that
$$
\limsup_{x\to 0+}F(x)x^3=1
$$
and
$$
\liminf_{x\to 0+}F(x)x=1,
$$
i.e., (\ref{F-r-x}) does not hold. To show that $F$ satisfies $(F3)$ one can use Fa\`a di Bruno's formula.
}}
\end{examples}
The main results of the paper are the following Theorems \ref{tri**}--\ref{b}.
\begin{theorem}
\label{tri**} For  $0<p\le \infty$ and
$\omega=\omega_1\ldots\omega_s$ such that $\omega_i\in\Omega$,
$i=1,\ldots,s$, the Bernstein inequality
\begin{equation}
\label{*} \|T'_n\|_{L_p(\omega u)} \le \, C\,n\, \|T_n\|_{L_p(\omega
u)}
\end{equation}
holds for any trigonometric polynomial of degree at most $n$ with $C=C(\omega, u, p)$, whenever $u$ is doubling if
$p<\infty$, and $u$ satisfies the $A^*$ condition if $p=\infty$.
\end{theorem} For example,
inequality~\eqref{*} holds for the following weight:
$$
\omega(t)=\exp(-1/\sin^2 t-1/\cos^4t).
$$
To prove Bernstein's inequality~\eqref{*} in the case when $\omega=\omega_1\in\Omega$, i.e., $s=1$, we use approximation properties of $\omega$.
To verify \eqref{*}
  with the product of weights each of which is from the class $\Omega$,
we need a new technique based on introduction of
 weighted classes for which Bernstein and Remez inequalities hold.
In particular, $\omega_i\in\Omega$ and $u$ as in Theorem \ref{tri**} belong to these classes.
This technique is developed in Sections \ref{section-remez} and \ref{section-lp}.

A necessary condition for Bernstein's inequality is given by the following result.

\begin{theorem}
\label{neg1}
 Let $\omega\in C(\T)$ be an arbitrary  weight
function satisfying the following conditions: $\omega \searrow$ on
$(-\epsilon,0)$, $\omega(0)=0$, $\omega \nearrow$ on
$(0,\epsilon)$, and moreover,
\begin{equation}\label{lim}
 \limsup_{t\to 0}\frac{\log{\omega(rt)}}{\log{\omega(t)}}=\infty
\quad\text{for some } \; r\in (0,1).
\end{equation}
Then for each $0<p\le\infty$ there exist a sequence of positive
integers $K_n\to\infty$ as $n\to\infty$, and a sequence of
trigonometric polynomials $Q_n$ of degree at most $K_n$ such that
$$
\lim_{n\to\infty}\frac{\|Q_n'\|_{L_p(\omega)}}{K_n\|Q_n\|_{L_p(\omega)}}=\infty.
$$
\end{theorem}

Theorem~\ref{tri**} and Theorem~\ref{neg1} provide a sharp condition
on the growth properties of a weight $\omega$ near the origin.
Specifically, if a weight $\omega$,
 satisfying  $\omega \searrow$ on $(-\epsilon,0)$, $\omega(0)=0$, $\omega
\nearrow$ on $(0,\epsilon)$, is such that Bernstein's inequality~(\ref{berw}) holds, then
 $\omega$
necessarily satisfies the following condition: for all $r\in(0,1)$,
\begin{equation}\label{limL}
\limsup_{t\to 0} \frac{\log{\omega(rt)}}{\log{\omega(t)}}=L<\infty.
\end{equation}
On the other hand, any $\omega\in\Omega$ satisfies~\eqref{limL}. Moreover, for each $r\in(0,1)$ and $L>1$, the weight $\omega(t)=\exp
(-|\sin t|^{-\alpha})$ fulfills (\ref{limL})  with $\alpha=-\log_rL$.
Then
 $\omega\in \Omega$ and by Theorem~\ref{tri**}  Bernstein's inequality~(\ref{berw}) holds for this weight.

If in (\ref{lim}) the limit (not only the limit superior) exists,
then a stronger result is true:
\begin{theorem}\label{b}
 Let $\omega\in C(\T)$ be an arbitrary  weight
function satisfying the following conditions: $w \searrow$ on
$(-\epsilon,0)$, $\omega(0)=0$, $\omega \nearrow$ on
$(0,\epsilon)$ and moreover,
\begin{equation*}
\lim_{t\to 0} \frac{\log{\omega(rt)}}{\log{\omega(t)}}=\infty
\quad\text{for each } r\in (0,1).
\end{equation*}
Then for each $0<p\le\infty$ there exists a sequence of
trigonometric polynomials $Q_n$
 of degree at most $n$ such that
$$
\lim_{n\to\infty}\frac{\|Q_n'\|_{L_p(\omega)}}{n\|Q_n\|_{L_p(\omega)}}=\infty.
$$
\end{theorem}

The paper is organized as follows. In Section \ref{section-growth} we discuss 
growth properties of weights from the class $\Omega$ which we will use further on.
Section \ref{section-approximation} presents the order of trigonometric approximation of functions from $\Omega$ as well as their
derivatives. In Section \ref{section-l1} we give the proof of Bernstein's
inequality with $\Omega$ weights in $L_1$.
 We will use it as a model case to prove
 the general Bernstein inequality (\ref{*}) in Section \ref{section-lp}.

In Section \ref{section-remez} we establish weighted Remez inequalities for trigonometric and algebraic polynomials.
Section \ref{section-lp} gives the proof of the general Bernstein inequality for $p\in (0,\infty]$.
 Theorem \ref{tri**} is a corollary of this result.
In Section \ref{section-algebraic} we study the weighted Bernstein and Markov inequalities for
algebraic polynomials on $[-1,1]$.  Section \ref{section-nikolskii} provides weighted Nikolskii's inequalities for trigonometric and algebraic polynomials.

Finally, in Section \ref{section-necessary} we prove a necessary condition for  Bernstein's
inequality~\eqref{berw} to hold. Namely, we verify Theorems
\ref{neg1} and \ref{b} as well as a result on the sharpness of
Theorem~\ref{neg1}.

Concerning algebraic polynomials on $[-1,1]$, it is important to
mention that for weights from the  class $\mathcal{W}$ the  Markov-Bernstein inequalities were obtained by Lubinsky and Saff
(cf. \cite{lub-saff} and the book \cite{levin}), see discussion in Section \ref{section-algebraic}.
A typical example of weights from the class $\mathcal{W}$ is
$\omega_{\alpha}(x)=\exp{\big(-(1-x^2)^\alpha\big)},\, \alpha>0$.
 We note that
 using \cite{lub-saff} one can also derive weighted Bernstein's inequality
\begin{equation*}
\|\sqrt{1-x^2}P'_n(x)\omega_{\alpha}(x)\|_{L_{\infty}[-1,1]} \le
C(\alpha)\, n \|P_n(x)\omega_{\alpha}(x)\|_{L_{\infty}[-1,1]},
\end{equation*}
 see Remark \ref{remark-lub}.
 We also note
  that  Bernstein's inequalities for algebraic polynomials were recently proved in \cite{in2, in1} for the weight
$\omega=\omega_\alpha u$, where $u$ is doubling. In Section \ref{section-algebraic} we deal with a more general class of weights. Our proof for the algebraic case is based on Bernstein's inequality for trigonometric polynomials from Section \ref{section-lp}.


By $C, C_i$  $(c, c_i)$ we will denote positive large (small,
respectively) constants that may be different on different
occasions. Also, below we will write that $C(\omega)$ if
$C(A,A_1,A_2,B,D)$, where $A,A_1,A_2,B,D$ are from the definition of
the class $\Omega$. Moreover, for the positive sequences $\{a_n\}$
and $\{b_n\}$,  $a_n\asymp b_n$ means that
$$
c\le \frac{a_n}{b_n}\le C.
$$

\vspace{6mm}

\section{Growth properties of  $\Omega$-functions}\label{section-growth}
Let $F:[-A,A]\setminus\left\{0\right\}\to (0,\infty)$ be an even
$C^{\infty}$ function on $(0,A]$ satisfying ($F_1$)--$(F_4).$ 

\begin{definition}
 For each $n\ge F(A)$ we denote by $x_0(n)$ a
unique positive solution of the equation $$F(x)=n.$$
\end{definition}

\begin{definition}
 For each $n\ge F(A)/A$ we denote by $x_1(n)$ a
unique positive solution of the equation $$F(x)=nx.$$
\end{definition}
Note that both sequences $\{x_0(n)\}$ and $\{x_1(n)\}$ are decreasing. 
\begin{lemma}
\label{Lemma1} There exist positive constants $C=C(A,A_1,A_2)$ and
$c=c(A,A_1,A_2)$ such that
$$
cx^{-A_2}<F(x)<Cx^{-A_1}, \quad x\in(0,A].
$$
\end{lemma}
\begin{proof}
By property (F4) we have
$$
|F'(x)|=-F'(x)\ge A_2\frac{F(x)}x,\quad x\in(0,A].
$$
Therefore,
$$
\log{F(x)}-\log{F(A)}=\int_x^A-(\log{F(t)})'dt\ge
\int_x^A\frac{A_2}tdt=A_2(\log{A}-\log{x}),
$$
which yields
$$
F(x)\ge F(A)A^{A_2}x^{-A_2}.
$$
Similarly, the inequality
$$
|F'(x)|\le A_1\frac{F(x)}x,\quad x\in(0,A],
$$
implies
$$
F(x)\le F(A)A^{A_1}x^{-A_1}.
$$
\end{proof}
\begin{lemma}
\label{Lemma2} For each $R>0$ there exist positive constants
$C=C(R,A_1,A_2)$ and $c=c(R,A_1,A_2)$ such that
$$
cx_0(Rn)<x_0(n)<Cx_0(Rn)
$$
for all $n$  large enough.
\end{lemma}
\begin{proof}
Let us prove the lemma for $R\ge 1$. For $R<1$ the proof is similar.
Since $F$ is decreasing on $(0,A]$ we take $c=1$. So, it is
enough to show that $x_0(n)<C x_0(Rn)$. By definition of $x_0(n)$
and $(F4)$ we have
\begin{eqnarray*}
(R-1)n&=&|F(x_0(n))-F(x_0(Rn))|=\int_{x_0(Rn)}^{x_0(n)}-F'(t)dt
\\
&\ge& A_2 \int_{x_0(Rn)}^{x_0(n)}F(t)\frac{dt}t\ge
A_2n(\log{x_0(n)}-\log{x_0(Rn)}).
\end{eqnarray*}
Thus, one may choose $C=\exp\left((R-1)/A_2\right)$.
\end{proof}
\begin{lemma}
\label{Lemma2.5} There exists a positive constant
$\alpha=\alpha(A,A_1,A_2)$  such that
$$
nx_1(n)\ge n^\alpha
$$
for all $n$  large enough.
\end{lemma}
\begin{proof}
Since $F(x_1(n))=nx_1(n)$ the proof of the lemma immediately follows
from Lemma~\ref{Lemma1}.
\end{proof}
By monotonicity of $F$ we have $x_1(2n)<x_1(n)$ and hence,
$2nx_1(2n)=F(x_1(2n))>F(x_1(n))=nx_1(n)$. In other words,
$x_1(2n)<x_1(n)<2x_1(2n)$. However, the following stronger statement holds.
\begin{lemma}
\label{Lemma3} There exists a positive constant $\epsilon=\epsilon(A,A_1,A_2)$ such that
$$
(1+\epsilon)x_1(2n)<x_1(n)<(2-\epsilon)x_1(2n)
$$
for all $n$ large enough.
\end{lemma}
\begin{proof}
Note that both $x_0(n)$ and $x_1(n)$ are monotonically decreasing to zero.
First, we show that
$$
(1+\epsilon)x_1(2n)<x_1(n).
$$
By definition of $x_1(n)$ and $(F4)$ we have
\begin{align*}
&2nx_1(2n)-nx_1(n)=F(x_1(2n))-F(x_1(n))=\int_{x_1(2n)}^{x_1(n)}-F'(t)dt
\\
&\le
A_1\int_{x_1(2n)}^{x_1(n)}\frac{F(t)}tdt\le
A_1(x_1(n)-x_1(2n))\frac{F(x_1(2n))}{x_1(2n)}=2nA_1(x_1(n)-x_1(2n)).
\end{align*}
Hence,
\begin{equation}
\label{a1}(2+2A_1)x_1(2n)\le x_1(n)(2A_1+1).
\end{equation}
Similarly,
\begin{eqnarray*}
2nx_1(2n)-nx_1(n) 
&\ge&
A_2\int_{x_1(2n)}^{x_1(n)}\frac{F(t)}tdt
\\&\ge&
A_2(x_1(n)-x_1(2n))\frac{F(x_1(n))}{x_1(n)}=nA_2(x_1(n)-x_1(2n)),
\end{eqnarray*}
which gives
\begin{equation}
\label{a2}(1+A_2)x_1(n)\le(2+A_2)x_1(2n).
\end{equation}
Finally, by~\eqref{a1} and~\eqref{a2} we take
$$
\epsilon=\frac 12\min{\left\{\frac
1{1+2A_1},\frac{A_2}{1+A_2}\right\}}.
$$
\end{proof}
\begin{corollary} \label{Cor4.1}
For each $C>0$ there exists $K=K(C,A,A_1,A_2)$ such that
\begin{equation}
\label{100} Cx_1(n)<Kx_1(Kn)
\end{equation}
for all $n$ large enough.
\end{corollary}
\begin{proof}
By Lemma~\ref{Lemma3} (the right-hand side estimate) there exists a positive constant $\delta=\delta(A,A_1,A_2)$ such that
 $2x_1(2n)>(1+\delta) x_1(n)$. Take an integer $m$ such that $(1+\delta)^m>C$.
 Then 
$$
2^mx_1(2^mn)>(1+\delta)^mx_1(n)>Cx_1(n),
$$
which is~\eqref{100} with $K=2^m$.
\end{proof}
Similarly, using Lemma~\ref{Lemma3} (the left-hand side estimate), we get
\begin{corollary}
\label{Cor4.1.5} For each $L>0$ there exists $Q=Q(L,A,A_1,A_2)$ such
that
\begin{equation}
\label{100.5} x_1(Qn)<\frac{x_1(n)}{L}
\end{equation}
for all $n$ large enough.
\end{corollary}
\begin{corollary}
\label{Cor4.2} For each $K>0$ there exists $L=L(K,A,A_1,A_2)$ such
that
\begin{equation}
\label{100.75} F\Big(\frac{x_1(n)}L\Big)>Kx_1(n)n
\end{equation}
for all $n$ large enough.
\end{corollary}
\begin{proof}
First, by~\eqref{100}
there exists $L=L(K,A,A_1,A_2)$
such that
$
Lnx_1(Ln)>Knx_1(n).
$
Second, on account of monotonicity of $F$,
$$
\frac{x_1(n)}L\le x_1(Ln),\quad L\ge 1.
$$
Therefore,
$$
F\left(\frac{x_1(n)}L\right)\ge F(x_1(Ln))=Lnx_1(Ln)>Knx_1(n).
$$
\end{proof}

\vspace{6mm}

\section{Approximation of $\Omega$-functions}
\label{section-approximation}
The aim of this section is to obtain an order of approximation of functions from the class $\Omega$ by trigonometric polynomials.

\subsection{Estimates for the Fourier coefficients of $\omega\in \Omega$}
 We use the classical estimate for the $n$-th Fourier coefficient of $\omega:$
\begin{equation}
\label{mumumu}
\big|\hat{\omega}_n\big|=\Big|\frac{1}{\pi}\int_{\T}{\omega}(t) \cos nt \,dt\Big|
\le2\inf_{k\ge 0}\frac{\|\omega^{(k)}\|_{C(\T)}}{n^k},\quad
n\ge 1.
\end{equation}
Below we obtain a uniform upper bound of the $n$-th derivative of the function $\omega\in \Omega$, where $\omega(t)=H(g(t))$, $H(x)=\exp{(-F(x))}$. To this end,
we use  Fa\`a di Bruno's formula
\begin{equation}
\label{101}
\big(u(v(x))\big)^{(k)}=\sum{\frac{k!}{m_1!\ldots
m_k!}}u^{(m_1+\ldots+m_k)}(v(x))\left(\frac{v'(x)}{1!}\right)^{m_1}
\ldots\left(\frac{v^{(k)}(x)}{k!}\right)^{m_k},
\end{equation}
where summation is taken over all nonnegative integers such that
$m_1+\ldots+km_k=k$.
 We start with the following technical lemma.
\begin{lemma} \label{Lemma0}
For each $k\in\N$ the following identity holds:
\begin{equation}\label{c1}
\sum_{m_1+2m_2+\ldots+km_k=k}\frac{k!}{m_1!m_2!\ldots m_k!(k-m_1-\ldots-m_k)!} = \frac 12 {{2k}\choose{k}}.
\end{equation}
\end{lemma}

\begin{proof}
Denote the left-hand side of (\ref{c1}) 
 by $S_k$. One can see that
$S_k$ is  the coefficient of $x^k$ of the polynomial
$$
(1+x+x^2+\ldots+x^k)^k,
$$
and hence, it is equal to the coefficient of $x^k$ in the Taylor series expansion of
the function
$$
f(x)=\frac 1{(1-x)^k}.
$$
Therefore,
$$
S_k=\frac{f^{(k)}(0)}{k!}=\frac 12 {{2k}\choose{k}}.
$$
\end{proof}
Now we are ready to estimate the maximum norm of the $k$-th derivative of the function $H$.
\begin{lemma}
\label{Lemma4.1} Let $H(x)=\exp{(-F(x))}$, where $F$  satisfies
$(F1)-(F4)$. Then $H\in C^{\infty}[-A,A]$, and there exists
$C=C(A,B,A_1,A_2)>0$ such that for all $k> F(A)$
$$
H^{(k)}(x)\le\left( \frac{Ck}{x_0(k)}\right)^k,\quad x\in[-A,A].
$$
\end{lemma}
\begin{proof}
Consider $x\in(0,A]$. By~\eqref{101},
$$
H^{(k)}(x)=\sum_{m_1+\ldots+km_k=k}{\frac{k!}{m_1!\ldots
m_k!}}(-1)^{m_1+\ldots+m_k}\exp{(-F(x))}\left(\frac{F'(x)}{1!}\right)^{m_1}
\ldots\left(\frac{F^{(k)}(x)}{k!}\right)^{m_k}.
$$
By $(F3)$ we have
$$
\frac{|F^{(s)}(x)|}{s!}\le C^s\frac{F(x)}{x^s}, \qquad 1\le s\le k.
$$
Hence,
\begin{eqnarray}
\nonumber|H^{(k)}(x)|&\le& C^k\sum_{m_1+\ldots+km_k=k}{\frac{k!}{m_1!\ldots
m_k!}}\frac{H(x)(F(x))^{m_1+\ldots+m_k}}{x^k}
\\
\label{102} &=&
C^k\sum_{m_1+\ldots+km_k=k}{\frac{k!}{m_1!\ldots
m_k!}}G_{m,k}(x),
\end{eqnarray}
where $m=m_1+\ldots+m_k$, and
$$
G_{m,k}(x):=\frac{H(x)(F(x))^m}{x^k}, \quad x\in(0,A].
$$
To estimate the maximum of $G_{m,k}(x)$ for $x\in(0,A)$, we write
$$
G'_{m,k}(x)=\frac{H(x)(F(x))^{m-1}}{x^k}\left(
-F'(x)F(x)+mF'(x)-\frac kxF(x)\right).
$$
Therefore, if $F(x)<k/A_1$, then $G'_{m,k}<0$.
Indeed,  by $(F4)$, we
get
$$
-F'(x)F(x)+mF'(x)-\frac kxF(x)<F(x)\left(-F'(x)-\frac
kx\right)<F(x)\left(\frac{A_1F(x)}x-\frac kx\right)<0.
$$
Similarly, if $F(x)>\max\{2,2/A_2\}k $, then $G'_{m,k}>0$. In this case
 $F(x)>2k\ge 2m$ and therefore,
\begin{align*}
&-F'(x)F(x)+mF'(x)-\frac kxF(x)\ge
-\frac{F'(x)F(x)}2-\frac kxF(x)
\\
&=\frac{F(x)}2\left(-F'(x)-\frac{2k}x\right)
\ge
\frac{F(x)}2\left(\frac{A_2F(x)}x-\frac{2k}x\right)
>0.
\end{align*}
Using the fact that each $G_{m,k}$ is a continuously differentiable
function on $(0,A]$, we get that $\max_{0<x\le A}{G_{m,k}(x)}$ exists
for all $1\le m\le k$ 
 and is attained at a point $x^*$ such that
\begin{equation}
\label{a3} \frac k{A_1}\le F(x^*)\le\max\{2,2/A_2\}k.
\end{equation}
Now, Lemma~\ref{Lemma1} implies that
$$
G_{m,k}(x)\to 0\qquad\mbox{as}\qquad 
 x\to 0+.
$$
Then, it follows from (\ref{102}) that $H^{(k)}(0)=0$ for all $k\in\N$, and hence $H\in C^{\infty}[-A,A]$.
Set $R=\max\{2,2/A_2\}$.
Since $F$ is decreasing, then we get by~\eqref{a3} that
$$
G_{m,k}(x)\le\exp{(-F(A))}\frac{(Rk)^m}{x_0^k(Rk)}\le\frac{C^kk^m}{x_0^k(k)},\qquad k>F(A).
$$
Here the last inequality follows from Lemma~\ref{Lemma2}.
Combining the latter with~\eqref{102} we obtain that
$$
|H^{(k)}(x)|\le\frac{
C^k k!}{x_0^k(k)}\sum_{m_1+\ldots+km_k=k}\frac{k^{m_1+\ldots+m_k}}{m_1!\ldots
m_k!}.
$$
Finally, taking into account that $k^{k-m}\ge (k-m)!$ for $1\le m\le
k$, we get by~\eqref{c1}
$$|H^{(k)}(x)|\le\frac{C^k k!}{x_0^k(k)}\sum_{m_1+\ldots+km_k=k}\frac{k!}{m_1!\ldots
m_k!(k-m_1-\ldots-m_k)!}\le\frac{C^k k!}{x_0^k(k)},\quad x\in(0,A].
$$
For $x\in[-A,0)$ the same inequality holds because $F$ and hence $H$ are even.
\end{proof}

We are now in a position to give the uniform estimate of
$\omega^{(k)}$, where $\omega\in \Omega$.
\begin{lemma}
\label{Lemma4.2} Let $\omega\in \Omega$, then there exists
$C=C(\omega)>0$ such that for all $k$ large enough
$$
\omega^{(k)}(t)\le\frac{C^k k^k}{x_0^k(k)},\quad t\in\mathbb{T}.
$$
\end{lemma}
\begin{proof}
Take $k\ge F(A)$ so that $x_0(k)$ is well defined.
Since $g$ is an analytic function on $\mathbb{T}$, and $H\in C^{\infty}[-A,A]$, then $\omega\in C^{\infty}(\mathbb{T})$.
By Fa\`a di Bruno's formula, it follows that, for each $k\in\N$,
$$
|\omega^{(k)}(t)|=|(H(g(t)))^{(k)}|=\sum_{m_1+\ldots+km_k=k}
{\frac{k!}{m_1!\ldots
m_k!}}H^{(m)}(g(t))\left(\frac{g'(t)}{1!}\right)^{m_1}
\ldots\left(\frac{g^{(k)}(t)}{k!}\right)^{m_k},
$$
where $m=m_1+\ldots+m_k$.
We rewrite the last sum as $\sum_{m<F(A)} + \sum_{m\ge F(A)}$.
Since  $H^{(m)}(x)\le C(\omega)$ for any  $m< F(A)$,
we have
\begin{equation}
\label{g-vspom-1}
\sum_{{m<F(A)}}
\le
 C(\omega) D^k
\sum_{{m<F(A)}}
{\frac{k!}{m_1!\ldots m_k!}}
\le
 C^k k!
\sum_{{m<F(A)}}
{\frac{1}{m_1!\ldots m_k!}}.
\end{equation}
To estimate $\sum_{m\ge F(A)}$, we can use 
\begin{equation}
\label{g-vspom}
H^{(m)}(x)\le\left( \frac{Cm}{x_0(m)}\right)^m,\quad m\ge F(A),
\end{equation}
provided by Lemma~\ref{Lemma4.1}
and~\eqref{g} to get
\begin{equation*}
\sum_{m\ge F(A)}\le
 D^k \sum_{m_1+\ldots+km_k=k}{\frac{k!}{m_1!\ldots
m_k!}}\left( \frac{C m}{x_0(m)}\right)^m.
\end{equation*}
Combining this with (\ref{g-vspom-1}), we get
$$
|\omega^{(k)}(t)|\le
\frac{C^k k!}{x_0^k(k)}\sum_{m_1+\ldots+km_k=k}{\frac{m^m}{m_1!\ldots
m_k!}},\quad t\in\mathbb{T}.
$$
Noting that for each integers $1\le m\le k$
$$
m^m\le\frac{k^k}{(k-m)!}\le\frac{C^kk!}{(k-m)!},
$$
we have by~\eqref{c1}
$$
|\omega^{(k)}(t)|\le\frac{C^k k!}{x_0^k(k)}
\sum_{m_1+\ldots+km_k=k}\frac{k!}{m_1!\ldots
m_k!(k-m_1-\ldots-m_k)!}\le\frac{C^k k!}{x_0^k(k)},\qquad
t\in\mathbb{T}.
$$
\end{proof}


The next result provides a near optimal $k$ in  estimate (\ref{mumumu}) for the $n$-th Fourier coefficient of $\omega\in \Omega$.
\begin{lemma}
\label{Lemma4.3} Let $F$ be a function satisfying $(F1)-(F4)$. Then
for each $C>e$ and $n$ large enough there exists an integer
$k=k(C,n,F)$ such that
$$
\frac{C^kk^k}{n^kx_0^k(k)}\le\exp{\Big(-\frac{1}{C^2}nx_1(n)+1\Big)}.
$$
\end{lemma}
\begin{proof}
Let $k$ be the minimal integer such that
\begin{equation}
\label{a4} \frac{Ck}{nx_0(k)}>\frac 1e.
\end{equation}
If $k<nx_1(n)/C^2$, then
$$
\frac{Ck}{nx_0(k)}<\frac 1C\frac{nx_1(n)}{nx_0(\frac
1{C^2}nx_1(n))}<\frac 1C\frac{x_1(n)}{x_0(nx_1(n))}=\frac 1C,
$$
where in the last equation we used the definitions of $x_0(n)$ and
$x_1(n)$. This
contradicts~\eqref{a4}. Thus, $k\ge nx_1(n)/C^2$. Finally, applying
again~\eqref{a4} we get
$$
\left(\frac{C(k-1)}{nx_0(k-1)}\right)^{k-1}\le\left(\frac
1e\right)^{\frac{1}{C^2}nx_1(n)-1},
$$
and the claim easily follows.
\end{proof}

We will also need the following technical result.
\begin{lemma}
\label{Lemma4.4} For each $\omega\in \Omega$ and $c>0$ we have
$$
\sum_{v=n}^{\infty}\exp{(-cvx_1(v))}\le\exp{\big(-\frac c2
nx_1(n)\big)}
$$
for all $n$ large enough, i.e., for $n\ge n_0(\omega, c)$.
\end{lemma}
\begin{proof}
Indeed, since the sequence $nx_1(n)$ is increasing to infinity,
$$
\sum_{v=n}^{\infty}\exp{(-cvx_1(v))}=\sum_{s=0}^{\infty}\sum_{k=2^sn}^{2^{s+1}n-1}\exp{(-ckx_1(k))}
\le\sum_{s=0}^{\infty}n2^s\exp{(-c2^snx_1(2^sn))}.
$$
By Lemma~\ref{Lemma3} there exists $\epsilon=\epsilon(A,A_1,A_2)$
such that $2x_1(2n)\ge (1+\epsilon)x_1(n)$ for all $n$ large enough.
Then
$$
\sum_{v=n}^{\infty}\exp{(-cvx_1(v))}
\le\sum_{s=0}^{\infty}n2^s\exp{(-c(1+\epsilon)^snx_1(n))}:=\sum_{s=0}^{\infty}h_s.
$$
It is easy to check that, for $s\ge 0$ and $n\ge n_0(\omega, c)$,
$$
\frac{h_{s+1}}{h_s}\le 2\exp{(-c\epsilon nx_1(n))}\le\frac 12.
$$
Thus, Lemma~\ref{Lemma2.5} gives
$$
\sum_{v=n}^{\infty}\exp{(-cvx_1(v))}
\le 2h_0=2n\exp{(-cnx_1(n))}\le \exp{\big(-\frac c2 nx_1(n)\big)},\quad
n\ge n_0(\omega, c).
$$
\end{proof}

\subsection{Remez inequality for trigonometric polynomials}
We will need the following Remez inequality answering how large can be $\|T_n\|_{L_\infty(\T)}$ if
$$
\left|\Big\{t\in \T: |T_n(t)|>1\Big\}\right|\le \varepsilon
$$
for some $0<\varepsilon\le 1$ holds.

\begin{lemma}\label{l1}\cite{erd}, \cite{er3}
 For any Lebesgue measurable set $B\subset \mathbb{T}$ such that $|B|< \pi/2$ we have
\begin{equation}\label{rem-or}
\|T_n\|_{L_\infty(\T)}\le \exp ({4 n|B|})
\|T_n\|_{L_\infty(\T \setminus B)}.
\end{equation}
If $0<p<\infty$ and $|B|< 1/4$ we have
\begin{equation}\label{rem-or-p}
\|T_n\|_{L_p(\T)}\le \Big(1+\exp ({4 n|B|p})\Big) \|T_n\|_{L_p(\T \setminus B)}.
\end{equation}
\end{lemma}

\subsection{Two approximation theorems for
the  $\Omega$-weights}

We are now ready to prove the following result on simultaneous trigonometric approximation of functions from the class $\Omega$ and their derivatives.
\begin{theorem}
\label{Th4.4} For each $\omega\in \Omega$ there exists a positive constant $c=c(\omega)$
 such that
\begin{equation}
\label{w11} \|\omega-\omega_n\|_{C(\mathbb{T})}\le\exp{(-cnx_1(n))}
\end{equation}
and
\begin{equation}
\label{w22}
\|\omega'-\omega'_n\|_{C(\mathbb{T})}\le\exp{(-cnx_1(n))}
\end{equation}
hold  for  $n$ large enough, where  $\omega_n$ is the $n$-th partial sum of the Fourier series of $\omega$.
\end{theorem}
\begin{proof}
Integrating by parts and Lemma~\ref{Lemma4.2} imply that, for some
$C>0$,
$$
|\hat{\omega}_n|
 \le 2 \frac{\|\omega^{(k)}\|_{C(\mathbb{T})}}{n^k}\le \frac{C^kk!}{x_0^k(k)n^k}.
$$
 Hence, by Lemma~\ref{Lemma4.3}, there exists $c=c(\omega)$ such that,
for $n\ge n_0(\omega)$,
$$
|\hat{\omega}_n|\le\exp{(-cnx_1(n))}.
$$
Let $\omega_n$ be the $n$-th partial sum of the Fourier series of
$\omega$, i.e.,
$$
\omega_n(t)=\frac{\hat{\omega}_0}{2}+\sum_{k=1}^n\hat{\omega}_k\cos{kt}.
$$
Since $\omega \in C^{\infty}(\mathbb{T})$ then for each
$t\in\mathbb{T}$
$$
\omega_n(t)\to\omega(t)\qquad\text{and}\qquad\omega'_n(t)\to\omega'(t)\quad\text{as}\quad
n\to\infty.
$$
Therefore, taking into account Lemma~\ref{Lemma4.4}, we have for each $t\in\mathbb{T}$
$$
|\omega(t)-\omega_n(t)|\le\sum_{v=n+1}^{\infty}|\hat{\omega}_v|\le\sum_{v=n+1}^{\infty}\exp{(-cvx_1(v))}\le\exp{(-\frac
c2 nx_1(n))}
$$
and
$$
|\omega'(t)-\omega'_n(t)|\le\sum_{v=n+1}^{\infty}v|\hat{\omega}_v|\le\sum_{v=n+1}^{\infty}\exp{(-\frac
c2 vx_1(v))}\le\exp{(-\frac c4 nx_1(n))}.
$$
\end{proof}
Let $g$ be an analytic function  such as in  the definition of the  class $\Omega$, i.e., satisfying  (\ref{g})
and such that each zero of $g$ is of multiplicity one.
Let $\{a_1,\ldots,a_m\}$ be the set of all zeros of $g$ on $\T$.
For each $\epsilon>0$ denote
$$
B_{\epsilon}:=\Big\{t\in\mathbb{T}:|g(t)|<\epsilon\Big\}.
$$
Let us show that  the measure of $B_{\epsilon}$ is at most linear in ${\epsilon}$.
\begin{lemma}
\label{Lemma4.5} For an arbitrary $\epsilon>0$ we have
$$
|B_{\epsilon}|\le \,C(g)\,\epsilon.
$$
\end{lemma}
\begin{proof}
Since all zeros of $g$ have multiplicity one, then
$$
|g(t)|=|(t-a_1)\ldots(t-a_m)h(t)|,
$$
where $\min_{t\in\mathbb{T}}|h(t)|=b(g)=:b>0$. Set
$$
S:=\left(\frac{3}{\min_{1\le i<j\le m}{|a_i-a_j|}}\right)^{m-1}.
$$
For given $\epsilon>0$, let $t_0\in\mathbb{T}$ be such that
$$
|t_0-a_i|>\frac{S\epsilon}b \qquad\text{for all }\quad i=\overline{1,m}.
$$
Since the inequality
$$
|t_0-a_j|\le\frac{\min_{1\le i<j\le m}{|a_i-a_j|}}3
$$
may hold at most for one $j\in\overline{1,m}$  we have
$$
|g(t_0)|\ge\frac{S\epsilon}b\left(\frac{\min_{1\le i<j\le
m}{|a_i-a_j|}}3\right)^{m-1}b =\epsilon.
$$
Hence, $t_0\not\in B_{\epsilon}$. Therefore, for each $t\in
B_{\epsilon}$, there exists $j\in\overline{1,m}$ such that
$$
|t-a_j|\le\frac{S\epsilon}b.
$$
Thus,
$$|B_{\epsilon}|\le\frac{2mS}b\,\epsilon.$$
\end{proof}


Now we are in a position to prove the following approximation theorem.
\begin{theorem}
\label{Th4.6} For each $\omega\in \Omega$ there exists an integer
constant $K=K(\omega)$ such that for each trigonometric polynomial $T_n$ we have
\begin{equation}
\label{approx} \frac 12\int_{\mathbb{T}}|T_n(t)\omega_{Kn}(t)|dt\le
\int_{\mathbb{T}}|T_n(t)|\omega(t)dt\le
2\int_{\mathbb{T}}|T_n(t)\omega_{Kn}(t)|dt,
\end{equation}
where $\omega_n$ is the $n$-th partial Fourier sum  of $\omega$.
\end{theorem}
\begin{proof}It is enough to verify (\ref{approx}) for sufficiently large $n$.
Using Theorem~\ref{Th4.4} we get
$$
\int_{\mathbb{T}}|T_n(t)||\omega(t)-\omega_{Kn}(t)|dt\le\exp{(-cKnx_1(Kn))}\int_{\mathbb{T}}|T_n(t)|dt.
$$
We define
$$
B_{x_1(n)}=\Big\{t\in\mathbb{T}:|g(t)|<x_1(n)\Big\}.
$$
Then, Lemma~\ref{Lemma4.5} implies that $|B_{x_1(n)}|\le Cx_1(n)$, where $C$ depends
only on $\omega$. Then, by the Remez inequality we get
\begin{eqnarray*}
\int_{\mathbb{T}}|T_n(t)||\omega(t)-\omega_{Kn}(t)|dt
&\le&
\exp{(-cKnx_1(Kn))}
\exp{(4n|B_{x_1(n)}|)}\int_{\mathbb{T}\setminus B_{x_1(n)}}|T_n(t)|dt
\\
&\le&
\exp{(-cKnx_1(Kn)+Cnx_1(n))}\int_{\mathbb{T}\setminus
B_{x_1(n)}}|T_n(t)|dt.
\end{eqnarray*}
Note that for each $t\in\mathbb{T}\setminus B_{x_1(n)}$,
\begin{equation}
\label{223}
\omega(t)=\exp{\left(-F(g(t))\right)}\ge\exp{\left(-F(x_1(n))\right)}=\exp{(-nx_1(n))}.
\end{equation}
Therefore,
$$
\int_{\mathbb{T}}|T_n(t)||\omega(t)-\omega_{Kn}(t)|dt\le\exp{\Big(-cKnx_1(Kn)+Cnx_1(n)+nx_1(n)\Big)}\int_{\mathbb{T}\setminus
B_{x_1(n)}}|T_n(t)|\omega(t)dt.
$$
Now, by Corollary~\ref{Cor4.1} we can choose integer $K$ large
enough such that
$$
\int_{\mathbb{T}}|T_n(t)||\omega(t)-\omega_{Kn}(t)|dt\le\frac
12\int_{\mathbb{T}\setminus B_{x_1(n)}}|T_n(t)|\omega(t)dt\le\frac
12\int_{\mathbb{T}}|T_n(t)|\omega(t)dt.
$$
This immediately implies the statement of the theorem.
\end{proof}

\vspace{6mm}
\section{Weighted Bernstein inequality in $L_1$}
\label{section-l1}
In this section we  prove the Bernstein inequality in $L_1(\omega)$, where $\omega\in \Omega$.
\begin{theorem}
\label{Th5.1.} Let $\omega\in \Omega$. Then for each trigonometric polynomial $T_n$ of degree at most $n$
\begin{equation}
\label{bern1}
\int_{\mathbb{T}}|T'_n(t)|\omega(t)dt\le\, C(\omega)\,n\int_{\mathbb{T}}|T_n(t)|\omega(t)dt.
\end{equation}
\end{theorem}
\begin{proof}
Since the inequality
\begin{equation}\label{continuous}
\int_{\mathbb{T}}|T'_n(t)|\omega(t)dt\le\,
C(\omega, n)\,\int_{\mathbb{T}}|T_n(t)|\omega(t)dt
\end{equation}
holds for any continuous weight $\omega$, it is enough
 to prove (\ref{bern1}) for $n$ large enough.
 The proof is in three steps.
 \\[10pt]
 \underline{Step 1}. By
Theorem~\ref{Th4.6} there exists an integer $K=K(\omega)$ large enough such that the $Kn$-partial Fourier sum $\omega_{Kn}$ satisfies
the following:
\begin{eqnarray}\label{dva}
\int_{\mathbb{T}}|T'_n(t)|\omega(t)dt
\nonumber
&\le&
2\int_{\mathbb{T}}|T'_n(t)\omega_{Kn}(t)|dt\\&\le&
2\int_{\mathbb{T}}|(T_n(t)\omega_{Kn}(t))'|dt+2\int_{\mathbb{T}}|T_n(t)\omega'_{Kn}(t)|dt=:I_1+I_2.
\end{eqnarray}
Then by the classical Bernstein inequality and Theorem~\ref{Th4.6} we
have
$$
I_1\le CKn\int_{\mathbb{T}}|T_n(t)\omega_{Kn}(t)|dt\le
C(\omega)n\int_{\mathbb{T}}|T_n(t)|\omega(t)dt.
$$
Further,
$$
I_2\le
2\int_{\mathbb{T}}|T_n(t)||\omega'(t)|dt+2\int_{\mathbb{T}}|T_n(t)||\omega'(t)-\omega'_{Kn}(t)|dt=:I_{21}+I_{22}.
$$
 \\[10pt]
\underline{Step 2}. To estimate $I_{21}$, let us define 
 the set
$$
B_{n,M}:=\Big\{t\in\mathbb{T}:g(t)\neq 0,\,\text{and
}|F'(g(t))g'(t)|\ge Mn\Big\}.
$$
Note that, for any $t\in B_{n,M}$, it follows from $(F4)$ that
$$
A_1\frac{F(g(t))}{|g(t)|}\|g'\|_{C(\mathbb{T})}\ge Mn,
$$
and therefore,
\begin{equation}
\label{111} \frac{F(g(t))}{|g(t)|}\ge M_2n,\quad \text{where}\quad
 M_2=\frac{M}{A_1D}.
\end{equation}
Using Corollary~\ref{Cor4.1.5} we have that for each $L>0$ there
exists $Q=Q(L,\omega)>1$ such that
$$
F\Big(\frac{x_1(n)}L\Big)\le F(x_1(Qn))=Qnx_1(Qn)<Qnx_1(n)
$$
for $n$ large enough.
Then, for all $x\in [x_1(n)/L,A]$, we get
\begin{equation}\label{vsp}
F(x)\le
F\Big(\frac{x_1(n)}L\Big)<Qnx_1(n)\le xQLn.
\end{equation}
Therefore, if
\begin{equation}
\label{225} M_2=\frac{M}{A_1D}>QL,
\end{equation}
then~\eqref{111} and \eqref{vsp} imply
$$
|g(t)|<\frac{x_1(n)}L, \quad t\in B_{n,M}.
$$
Now, for each $K\in\N$, taking $L=L(K,\omega)$ as in Corollary~\ref{Cor4.2} we have
\begin{equation}\label{zv17}
F(g(t))\ge F\Big(\frac{x_1(n)}L\Big) \ge Kx_1(n)n.
\end{equation}
Moreover, by Lemma~\ref{Lemma1} we have
$$
\frac{F(g(t))}{|g(t)|}\le C(\omega)\left(F(g(t))\right)^{1+\frac
1{A_2}},\quad t \in B_{n,M}.
$$
Let us estimate $|\omega'(t)|$ from above for $t\in B_{n,M}$. In view of $(F_1)$ and $(F_4)$, we get
\begin{eqnarray}\nonumber
|\omega'(t)|&=&\omega(t)\Big|F'(g(t))g'(t)\Big|\le A_1 D\,\omega(t)\frac{F(g(t))}{|g(t)|}
\\\nonumber&\le&
 C(\omega)\exp{\left(-F(g(t))\right)}\big(F(g(t))\big)^{1+\frac1{A_2}}
\\
\label{112} &\le& C(\omega)\exp{\left(-F(g(t))/2\right)},\qquad\qquad t\in
B_{n,M},
\end{eqnarray}
where in the last estimate we have used (\ref{zv17}) and the fact that $nx_1(n)\to \infty$ as $n\to \infty$.

Hence, \eqref{zv17} and \eqref{112} imply
\begin{equation}
\label{a5} |\omega'(t)|\le
C(\omega)\exp{\left(-Kx_1(n)n/2\right)},\quad t\in B_{n,M}.
\end{equation}
 \\[10pt]
\underline{Step 3}. Now we are ready to estimate $I_{21}$. We have
$$
I_{21}=
2\int_{B_{n,M}}|T_n(t)||\omega'(t)|dt+2\int_{\mathbb{T}\setminus
B_{n,M}}|T_n(t)||\omega'(t)|dt=:I_{211}+I_{212}.
$$
Let us estimate $I_{211}$. Thanks to \eqref{a5}, we obtain
$$
I_{211}=2\int_{B_{n,M}}|T_n(t)||\omega'(t)|dt\le
C(\omega)\exp{\left(-Kx_1(n)n/2\right)}\int_{B_{n,M}}|T_n(t)|dt
$$
$$
\le
C(\omega)\exp{\left(-Kx_1(n)n/2\right)}\int_{\mathbb{T}}|T_n(t)|dt.
$$
Now, as in the proof of Theorem~\ref{Th4.6}, we consider
$$
B_{x_1(n)}=\Big\{t\in\mathbb{T}:|g(t)|<x_1(n)\Big\}.
$$
By the Remez inequality and Lemma~\ref{Lemma4.5} we get
\begin{eqnarray}
I_{211}&\le&
C(\omega)\exp{\left(-Kx_1(n)n/2\right)}\exp{(4n|B_{x_1(n)}|)}\int_{\mathbb{T}\setminus
B_{x_1(n)}}|T_n(t)|dt
\nonumber
\\
\nonumber &\le&
C(\omega)\exp{\left(-Kx_1(n)n/2+C(\omega)nx_1(n)\right)}\int_{\mathbb{T}\setminus
B_{x_1(n)}}|T_n(t)|\omega(t)dt\\&\le&
\label{224}
C(\omega)\int_{\mathbb{T}}|T_n(t)|\omega(t)dt
\end{eqnarray}
for $K\in\N$ large enough. On the other hand, it follows from the definition of $B_{n,M}$ that
$$
I_{212}=2\int_{\mathbb{T}\setminus B_{n,M}}|T_n(t)||\omega'(t)|dt\le
2M\, n\int_{\mathbb{T}}|T_n(t)|\omega(t)dt.
$$
Thus,
$$
I_{21}\le C(\omega)n\int_{\mathbb{T}}|T_n(t)|\omega(t)dt.
$$
Regarding $I_{22}$, we first note that Theorem~\ref{Th4.4} yields
\begin{equation*}
I_{22}\le
\exp{(-c(\omega) Knx_1(Kn))}\int_{\mathbb{T}}|T_n(t)|dt.
\end{equation*}
Similarly as we proceed in the estimates of $I_{211}$, we use Remez's inequality for the set $B_{x_1(n)}$ and Lemma \ref{Lemma4.5} to get
\begin{equation}
\label{227} I_{22}\le C(\omega)\int_{\mathbb{T}}|T_n(t)|\omega(t)dt,
\end{equation}
for $K\in\N$ large enough.

Let us explain how we choose the constants $K, L,Q,$ and $M$. First, $K\in\N$ is taking large enough such that
~\eqref{dva},
~\eqref{224},
and~\eqref{227} hold. Further we choose
$L=L(K,\omega)$ as in Corollary~\ref{Cor4.2}, $Q=Q(L,\omega)$ as in
Corollary~\ref{Cor4.1.5}, and finally $M>QLA_1D$ so that~\eqref{225}
holds.
\end{proof}

\vspace{6mm}

\section{Weighted Remez inequalities}
\label{section-remez}
 For an arbitrary measurable set $E$, denote
$
\|T_n\|_{L_p(\omega, E)}=
\left(\int_{E}|T_n|^p\omega\right)^{1/p}$ if $p<\infty$ and
$
\|T_n\|_{L_\infty(\omega, E)}= \esssup_{t\in E}|T_n(t)\omega(t)|$.
We use the notation $\|T_n\|_{L_p(\omega)}$ rather than
$\|T_n\|_{L_p(\omega,\T)}$.

The following classes play an important role in our further study.
\begin{definition}
We say that a weight $u$ satisfies the $\mathcal{R}(p)$ condition,
$0<p\le \infty$, and write $u\in \mathcal{R}(p)$,
 if for any trigonometric polynomial $T_n$ the
weighted Remez inequality holds, that is, there exists $C=C(p,u)>0$ such that
\begin{equation}
\label{665-1} \|T_n\|_{L_p(u, \T)} \le \, \exp({C n |E|})\,
\|T_n\|_{L_p(u, \T\backslash E)}
\end{equation}
for all measurable sets $E$ with $|E|\le 1$.
\end{definition}

\begin{definition}
We say that a weight $u$ satisfies the $\mathcal{R}_{int}(p)$ condition, $0<p\le \infty$, and write $u\in \mathcal{R}_{int}(p)$,
 if for any trigonometric polynomial
$T_n$ the restricted weighted Remez inequality holds, that is, there exists $C=C(p,u)>0$ such that
\begin{equation} \label{665}
\|T_n\|_{L_p(u, \T)} \le \, \exp({C n |E|})\, \|T_n\|_{L_p(u,
\T\backslash E)}
\end{equation}
for all sets $E$ which are a finite union of intervals of length
$\ge 1/n$ and such that $|E|\le 1$.
\end{definition}
\begin{remark}
One can define the class $\mathcal{R}_{int}(p, d)$ such that for any $T_n$ and
the set $E$ being a finite union of intervals of length $\ge d/n$
we have
$\|T_n\|_{L_p(u, \T)} \le \, \exp({C n |E|})\, \|T_n\|_{L_p(u,
\T\backslash E)}$ for some constant $C=C(p, u, d)$.
Then $\mathcal{R}_{int}(p, d)= \mathcal{R}_{int}(p).$
\end{remark}
We will need the following approximation inequalities for the weight $\omega^{1/p}$ that are similar to Theorems \ref{Th4.4} and \ref{Th4.6}.

\begin{lemma} \label{lemma-v^p} Let $\omega=\exp \big({-F(g(t))}\big)\in \Omega$ and $v=\omega^{1/p}$ for $p\in (0, \infty)$.
Let $v_n$ is the $n$-th partial Fourier sum of $v$.
\\
{\textnormal{ (A)}}.
We have 
\begin{equation}\label{2zv}
 \|v^p-|v_n|^p\|_{C(\mathbb{T})}\le\exp{(-c(p,\omega)nx_1(n))}
\end{equation}
and\begin{equation}\label{3zv}
\|v'-v'_n\|_{C(\mathbb{T})}\le\exp{(-c(p,\omega)nx_1(n))}
\end{equation}
for $n$ large enough, where 
 $x_1(n)$ is the unique positive solution of the equation $F(x_1(n))=nx_1(n)$.
\\
{\textnormal{ (B)}}. For any $u\in \mathcal{R}_{int}(p)$, there exists
$K=K(\omega, u,p)$ such that
\begin{equation}\label{500}
\frac 12\int_{\mathbb{T}}|T_n(t)|^p |v_{Kn}(t)|^pu(t)dt \le
\int_{\mathbb{T}}|T_n(t)|^p \omega(t)u(t)dt\le
2\int_{\mathbb{T}}|T_n(t)|^p|v_{Kn}(t)|^pu(t)dt.
\end{equation}
\\
{\textnormal{ (C)}}. For any $u\in \mathcal{R}_{int}(\infty)$, there exists
$K=K(\omega, u)$ such that
\begin{equation}\label{5005}
\frac 12 \| T_n \omega_{Kn}  u \|_{L_\infty(\T)}
\le \| T_n \omega  u \|_{L_\infty(\T)}
\le
2 \| T_n \omega_{Kn}  u \|_{L_\infty(\T)},
\end{equation}
where $\omega_n$ is the $n$-th partial Fourier sum of $\omega$.
\end{lemma}
\begin{proof} We may assume that $n$ is large enough.
For any $\omega=\exp\big({-F(g(t))}\big)\in \Omega$ and for any $p\in
(0, \infty)$ we have, by definition of the class $\Omega$,
$$v(t)=\omega^{1/p}(t)=\exp\big({-H(g(t))}\big)\in \Omega, \qquad t\in\T,$$
where $H(x)=F(x)/p$ satisfies $(F_1)-(F_4)$. Moreover, by Corollary
\ref{Lemma2.5}
$$x_1^\omega(n)\asymp x_1^v(n),$$
where
 $x_1^\omega(n)$ is a
unique positive solution of the equation $F(x_1^\omega(n))=nx_1^\omega(n)$ and
 $x_1^v(n)$ is a
unique positive solution of the equation $H(x_1^v(n))=nx_1^v(n)$.

To verify (\ref{2zv}) and (\ref{3zv}), we use Theorem \ref{Th4.4} and the following inequality:
\begin{equation}\label{latter} \Big|v^p(t)-|v_{Kn}(t)|^p\Big|  \le C(p,
\omega) \Big|v(t)-v_{Kn}(t)\Big|^{\min\{1,p\}},\qquad 0<p<\infty,
\quad t\in \T.
\end{equation}
For $0<p<1$, the latter follows from the inequality $|a^p-b^p|\le
C(p)|a-b|^p,$ where $a,b\ge 0$. For $p>1$, we get~\eqref{latter}
using the fact that if $a>b>0$ then ${a^p-b^p}=p\xi^{p-1}(a-b)$ for
some $\xi\in (b,a)$. Thus, the proof of part {(A)} is complete.

To show {(B)} and {(C)}, we follow the proof of Theorem \ref{Th4.6} using (\ref{2zv}) and the following remark.

\begin{remark}\label{remark6.1}
{
\it{
In the proofs of Theorems \ref{Th4.4} and \ref{Th4.6}, we use the Remez inequalities  only for the set
$$
B_{x_1(n)}=\Big\{t\in \T:|g(t)|< x_1(n)\Big\}.
$$
Analyzing the proof of Lemma \ref{Lemma4.5}, we note that
there exists $\widehat{B}_{x_1(n)}\subset\T$ such that ${B_{x_1(n)}}\subseteq \widehat{B}_{x_1(n)}$, $|\widehat{B}_{x_1(n)}|\le C x_1(n)$,
and
$\widehat{B}_{x_1(n)}$ is a union of $m$  intervals of length $>1/n$, where $m$ is a number of zeros of $g$ on $\T$.
Therefore, in the proof of Theorems \ref{Th4.4} and \ref{Th4.6} we can apply the Remez inequality for the set $\widehat{B}_{x_1(n)}$.
}}
\end{remark}
\end{proof}

In this section we prove the following general Remez inequality in $L_p$.
\begin{theorem}\label{remez-theorem}
 Let $0<p\le\infty$, $\omega\in \Omega$, and $u\in \mathcal{R}(p)$.  Then for each trigonometric polynomial $T_n$ we have
\begin{equation}
\label{remez*} \|T_n\|_{L_p(\omega u)} \le \, \exp({C n |E|})\,
\|T_n\|_{L_p(\omega u, \T\backslash E)},
\end{equation}
where $C=C(\omega, u, p)$ and $E$ is a measurable set of positive
measure $|E|\le 1$.
\end{theorem}

Since any $A_\infty$  weight $u$ satisfies the $\mathcal{R}(p)$
condition, $0<p<\infty$ (see \cite[Th. 5.2]{mas1} and \cite[Th.
7.2]{erd}) and any $A^*$-weight $u$ satisfies the
$\mathcal{R}(\infty)$ condition (see \cite[(6.10)]{mas1}), Theorem
\ref{remez-theorem} immediately implies the following result.

\begin{corollary}\label{corcor}
For  $0<p<\infty$, the Remez inequality (\ref{remez*}) holds for any
measurable set  $E$, $|E|\le 1$ whenever  $\omega\in \Omega$ and
$u\in A_\infty$ and for $p=\infty$, whenever $\omega\in \Omega$ and $u\in A^*$. Moreover, applying
Theorem~\ref{remez-theorem} several times we obtain
inequality~\eqref{remez*} for the weight
$\omega=\omega_1\ldots\omega_s$, where $\omega_i\in\Omega$,
$i=1,\ldots,s$.
\end{corollary}

 Conditions on the weight $u$ in Corollary \ref{corcor} can be relaxed in the case when  the set $E$ is a finite union of intervals.
 First, we give an analogue of Theorem \ref{remez-theorem} for this case.

\begin{theorem}\label{remez-theorem-1}
 Let $0<p\le\infty$, $\omega\in \Omega$, and $u\in \mathcal{R}_{int}(p)$.  Then for each trigonometric polynomial $T_n$ we have
\begin{equation}
\label{remez***} \|T_n\|_{L_p(\omega u)} \le \, \exp({C n |E|})\,
\|T_n\|_{L_p(\omega u, \T\backslash E)},
\end{equation}
where $C=C(\omega, u, p)$ and $E$ is a finite union of intervals of
length $\ge 1/n$.
\end{theorem}
In particular, this and \cite[Th. 5.3]{mas1} give a refinement  of
Corollary \ref{corcor} for such sets $E$.

\begin{corollary}
\label{remez-int} For  $0<p<\infty$ the Remez inequality (\ref{remez***}) holds whenever  $\omega\in \Omega$,  $u$ is
doubling, and $E$ is a union of intervals of length  $\ge 1/n$. Moreover, applying Theorem~\ref{remez-theorem-1} several
times we obtain inequality~\eqref{remez***} for the weight
$\omega=\omega_1\ldots\omega_s$, where $\omega_i\in\Omega$,
$i=1,\ldots, s$.
\end{corollary}

\begin{proof}[Proof of Theorem \ref{remez-theorem}]
It is sufficient to show (\ref{remez*}) for $n$ large enough.
 Let first $p\in (0,\infty)$.
It follows from Lemma \ref{lemma-v^p} that
 for $v=\omega^{1/p}\in \Omega$ we have
\begin{equation}\label{xixi}
 \|v^p-|v_n|^p\|_{C(\mathbb{T})}\le\exp{(-c(p,\omega)nx_1(n))},
\end{equation}
where $v_n$ is the $n$-th partial Fourier sum of the function $v$.
Moreover, by (\ref{500})
\begin{equation}\label{zv-new}
\int_{\T} |T_n|^p v^p\,u\, \asymp \int_{\T} |T_{n}|^p |v_{Kn}|^p\,u\
\end{equation}
for $K=K(\omega, u, p)$ large enough. Let us also remind that $$
B= B_{x_1(n)}=\Big\{t\in \T : \;|g(t)|\le x_1(n)\Big\}.
$$ \underline{Case 1.} Let $|B|\le |E|$.
Using (\ref{zv-new}) and (\ref{665-1}) for $u\in \mathcal{R}(p)$,
we have
\begin{eqnarray*}
\int_{\T} |T_n|^p \omega u&\le & \, \exp({C(p,u) K n (|E|+|B|)})\,
\int_{\T\setminus (E\cup B)} |T_n|^p |v_{Kn}|^p u
\\&\le&
\, \exp({C(p,u) K n |E|})\, \int_{\T\setminus (E\cup B)} |T_n|^p
|v_{Kn}|^p u.
\end{eqnarray*}
The latter can be estimated by $I_1+I_2$, where
\begin{eqnarray*}
I_1:= \exp({C(p,u) K n |E|})\, \int_{\T\setminus (E\cup B)} |T_n|^p
v^p u,
\end{eqnarray*}
and
\begin{eqnarray*}
I_2:= \exp({C(p,u) K n |E|})\, \int_{\T\setminus (E\cup B)} |T_n|^p
\Big|v^p -|v_{Kn}|^p\Big| u.
\end{eqnarray*}
Corollary \ref{Cor4.1} implies that, for any  $c>0$, there exists $K=K(c,\omega)$
such that $ x_1(n)< c Kx_1(Kn)$ and therefore $\exp({-cKn x_1(Kn) })\le
\exp({-n x_1(n)}).$ Then, by (\ref{xixi}) for $c=c(p,\omega)$,
\begin{eqnarray*}
\Big|v^p -|v_{Kn}|^p\Big| \le \exp({-cKn x_1(Kn) })\le \exp({-n x_1(n)})\le
\omega(t),
\qquad t\in \T\setminus B,
\end{eqnarray*}
 where the last inequality follows from (\ref{223}).
 Thus, $$I_1+I_2\le 2I_1\le 2
\exp({C(p,\omega, u)  n |E|})\, \int_{\T\setminus E} |T_n|^p \omega u.
$$
\underline{Case 2.} Let $|B|> |E|$. Similarly to  Case 1,  using
(\ref{665-1}), we get
\begin{eqnarray*}
\int_{\T} |T_n|^p \omega u&\le & I_1+I_2,
\end{eqnarray*}
 where
\begin{eqnarray*}
I_1:= \exp({C(p,u) K n |E|})\,\int_{\T\setminus E} |T_n|^p v^p u
\quad\text{and}\quad
I_2:= \exp({C(p,u) K n |E|})\, \int_{\T\setminus E} |T_n|^p \big|v^p
-|v_{Kn}|^p\big| \, u
.
\end{eqnarray*}
By (\ref{xixi}),
\begin{eqnarray*}
I_2\le  \exp({C(p,u) K n |E|})\, \exp({-c(p,\omega) Kn x_1(Kn) }) \int_{\T}
|T_n|^p  u.
\end{eqnarray*}
Applying again the Remez inequality (\ref{665-1}), we have
\begin{eqnarray*}
I_2\le  \exp({C(p,u) K n |E|})\, \exp({-c(p,\omega) Kn x_1(Kn) }) \exp({C(p, u)
n (|B|+|E|)}) \int_{\T\setminus (E\cup B)} |T_n|^p  u.
\end{eqnarray*}
Since $\omega(t)\ge \exp({- n x_1(n)})$, $t\in \T\setminus B$, we get
\begin{eqnarray*}
I_2&\le&  \exp({C(p,u) K n |E|})\, \exp({-c(p,\omega) Kn x_1(Kn) }) \exp({C(p,
u)  n |B|}) \exp({ n x_1(n)}) \int_{\T\setminus (E\cup B)} |T_n|^p
\omega u.
\end{eqnarray*}
Taking into account that $|B|\le C(\omega) x_1(n)$, we obtain that
$$\exp\Big({C(p,u) K n |E|} {-c(p,\omega) Kn x_1(Kn) }+ {C(p, u)  n |B|}+{ n x_1(n)}\Big)
 \le
\exp\Big(C(p,u) K n |E|\Big)
$$
 for $K=K(\omega,u,p)$ large enough. Thus,
\begin{eqnarray*}
I_2&\le& \,
 \exp(C(p,u) K n |E|)\,\int_{\T\setminus E} |T_n|^p
\omega u.
\end{eqnarray*}
Collecting estimates for $I_1$ and $I_2$, we arrive at
$$\int_{\T} |T_n|^p \omega u \le
\exp({C(p,\omega,u)  n |E|})\, \int_{\T\setminus E} |T_n|^p  \omega
u ,\qquad p\in (0,\infty),$$ which is the required inequality.

The proof in the case $p=\infty$ follows along the same lines as above and left for the reader.
\end{proof}

Proof of Theorem \ref{remez-theorem-1} is similar to the proof of
Theorem \ref{remez-theorem} thanks to Remark \ref{remark6.1}.

We now give the following important corollary of the Remez inequalities for the product of weights.
\begin{corollary}\label{cor-remez-remez}
Let $\omega=\omega_1\ldots\omega_s$, where $\omega_i\in\Omega$,
$i=1,\ldots,s$. Let also
  $0<p\le\infty$ and
  $u\in \mathcal{R}_{int}(p)$.
  Then
$$\int_\T|T_n|^p\omega u \asymp
\int_\T|T_n|^p |v^{(1)}_{Kn}|^p \cdots |v^{(s)}_{Kn}|^p u, \qquad 0<p<\infty,
$$
where $v^{(i)}_{n}$ is the $n$-th partial Fourier sum of
$\omega_i^{1/p}$, $i=1,\ldots,s$, and $K=K(\omega,u,p)$ is large
enough. Moreover,
$$
\|T_n \omega u \|_{L_\infty(\T)} \asymp \|T_n v^{(1)}_{Kn} \cdots
v^{(s)}_{Kn} u \|_{L_\infty(\T)}
$$
where $v^{(i)}_{n}$ is the $n$-th partial Fourier sum of $\omega_i$,
$i=1,\ldots,s$, and $K=K(\omega,u)$ is large enough.
\end{corollary}
To prove this, we use induction, Lemma \ref{lemma-v^p} {\bf (B)}, and 
 the following result provided by Corollary \ref{remez-int} for $p<\infty$ and Theorem \ref{remez-theorem-1} for $p=\infty$:
 if $\omega_i\in\Omega$ and $u\in \mathcal{R}_{int}(p)$, we have $\omega_1\ldots\omega_l u\in \mathcal{R}_{int}(p)$ for any integer $1\le l\le s-1$.

We finish this section by proving  the following Remez
inequality for algebraic polynomials $P_n$.

\begin{corollary}\label{corollary-remez}
 Let $0<p<\infty,$ $\omega=\omega_1\ldots\omega_s$, where
 $\omega_i(\cos t)\in \Omega$, $i=1,\ldots,s$.
Then the following inequality
\begin{equation}\label{vspom}
\|P_n\|_{L_p(\omega u, [-1,1])} \le \, \exp({C(p,\omega, u) n
\sqrt{|E|}})\, \|P_n\|_{L_p(\omega u, [-1,1]\backslash E)}
\end{equation}
holds for all measurable sets $E$ with $|E|\le 1/4$ and a weight
$u\in A_{\infty}$. For $p=\infty$, \eqref{vspom} holds for a weight
$u\in A^*$.
\end{corollary}
 \begin{proof}
To prove (\ref{vspom}), we use
change of variables $x=\cos t$, Corollary~\ref{corcor}, and the following two facts:
\begin{equation} \label{A-infty}
u\in A_{\infty}\quad \mbox{on}\; [-1,1] \quad \mbox{if and only if}\quad u(\cos t)|\sin t|\in
A_{\infty}\; \mbox{on}\; \T;
\end{equation}
 see \cite[p. 63]{mas1} and
 \begin{equation} \label{A-zvezda}
u\in A^*\quad \mbox{on}\; [-1,1] \quad \mbox{if and only if}\quad u(\cos t)\in A^*\; \mbox{on}\; \T;
\end{equation}
 see \cite[p. 68]{mas1}.

To conclude the proof, we remark that  for the map $\Phi(t)=\cos t$ and any measurable set $E\subset [-1,1]$ with $|E|\le 1/4,$ we get $|\Phi^{-1}(E)|\le 2 \sqrt{|E|}\le 1$.
\end{proof}

An analogue of Theorem \ref{remez-theorem-1} for algebraic polynomials can be written
similarly.

\vspace{6mm}

\section{Weighted Bernstein inequality in $L_p$}
\label{section-lp}

The goal of this section is to establish the weighted Bernstein inequality in $L_p$ for the case of product of weights generalizing Theorem \ref{Th5.1.}.
 The proof combines the approximation technique that we used
in Theorem \ref{Th5.1.} and the Remez inequalities from Section \ref{section-remez}.

\begin{definition}
We say that a weight $u$ satisfies the $\mathcal{B}(p)$ condition, $0<p\le \infty$, and write
$u\in \mathcal{B}(p)$,
if for any trigonometric polynomial $T_n$ of degree at most $n$ the weighted  Bernstein inequality holds, that is,
\begin{equation}
\label{cl-b}
\|T'_n\|_{L_p(u)}\le \,C(p,u)\, n \,\|T_n\|_{L_p(u)}.
\end{equation}
\end{definition}

\begin{theorem}
\label{Th5.1'}
Let $0<p\le\infty$, $\omega\in \Omega$, $u\in \mathcal{B}(p)\cap\mathcal{R}_{int}(p)$, then for any trigonometric polynomial $T_n$ of degree at most $n$ we have
\begin{equation}
\label{tri*}
\|T'_n\|_{L_p(\omega u)}
\le \,
C\,n\, \|T_n\|_{L_p(\omega u)},
\end{equation}
where $C=C(\omega, u, p)$.
\end{theorem}

\begin{proof}
It is enough to prove (\ref{tri*}) for $n$ large enough.
We start with the case $0<p<\infty$.

First, by (\ref{500}), we have that, for some $K=K(\omega,u,p)$,
\begin{eqnarray*}
\int_{\T} |T_n'|^p \omega\,u\,
&\le& 2 \int_{\T} |T_n'|^p  |v_{Kn}|^p\,u\,
\\
&\le& 2^{1+p}\left(
\int_{\T} \Big|(T_n v_{Kn})'\Big|^p u\, +
\int_{\T} \Big|T_n v_{Kn}'\Big|^p u\, \right).
\end{eqnarray*}
Since $u\in \mathcal{B}(p)$, we get
$$
\int_{\T} \Big|(T_n v_{Kn})'\Big|^p u
\le
 C(\omega, u,p) \,n^p
\int_{\T} \Big|T_n v_{Kn}\Big|^p u
\le
 C(\omega, u,p) \,n^p
\int_{\T} \Big|T_n\Big|^p \omega u.
$$
Also,
\begin{eqnarray*}
\int_{\T} \Big|T_n v_{Kn}'\Big|^p u\, \le
2^p\left(
\int_{\T} \Big|T_n v'\Big|^p u\, +
\int_{\T} \Big|T_n\Big|^p \Big| v'-v'_{Kn}\Big|^p u
\right).
\end{eqnarray*}
To conclude the proof, we follow estimates of $I_{21}$ and $I_{22}$ in the proof of Theorem \ref{Th5.1.} taking into account (\ref{3zv}).
Note that in view of Remark \ref{remark6.1} it suffices to assume $u\in \mathcal{R}_{int}(p)$.

Finally, we arrive at
\begin{eqnarray*}
\int_{\T} |T_n'|^p \, \omega\, u\,\le C(\omega, u, p) n^p\int_{\T} |T_n|^p \omega\, u.
\end{eqnarray*}

The proof for the case $p=\infty$ repeats the same lines as the proof in the case $0<p<\infty$
using Theorem \ref{Th4.4} and the inequality
\begin{equation}\label{rrr}
\frac12\| T_n \omega_{Kn} u\|_{L_{\infty(\T)}}\le \| T_n \omega
u\|_{L_{\infty(\T)}}\le 2 \| T_n \omega_{Kn} u\|_{L_{\infty(\T)}},
\end{equation}for $K$ large enough provided by Lemma \ref{lemma-v^p} (C).
First,
\begin{eqnarray*}
\| T_n' \omega u\|_{L_{\infty(\T)}}\le C \,\Big( \|(T_n
\omega_{Kn})' u\|_{L_{\infty(\T)}} + \|T_n \omega_{Kn}'
u\|_{L_{\infty(\T)}} \Big) \le C \,\Big( n \|T_n \omega_{Kn}
u\|_{L_{\infty(\T)}}+ \|T_n \omega_{Kn}' u\|_{L_{\infty(\T)}}\Big),
\end{eqnarray*}
where $C=C(\omega,u,p)$.
In view of (\ref{rrr}), $n \|T_n \omega_{Kn} u\|_{L_{\infty(\T)}}
\le 2 n \|T_n \omega u\|_{L_\infty(\T)}$. To estimate the second
term, we write
$$\|T_n \omega_{Kn}' u\|_{L_{\infty(\T)}}\le \Big(\esssup_{t\in B_{n,M}} + \esssup_{t\in \T\backslash B_{n,M}} \Big)|T_n(t) \omega_{Kn}'(t) u(t)|$$
and use Remez's inequality with $u\in \mathcal{R}_{int}(p)$ and Theorem \ref{Th4.4} to get $\|T_n
\omega_{Kn}' u\|_{L_{\infty(\T)}}\le C n \|T_n \omega
u\|_{L_\infty(\T)}$.
\end{proof}



Now we are in position to prove Theorem \ref{tri**} stated in Introduction.

\begin{proof}[Proof of Theorem \ref{tri**}]
First, any doubling weight $u$ satisfies Bernstein's inequality
(\ref{cl-b}) for $0<p<\infty$ (see \cite[Th. 4.1]{mas1} and
\cite[Th. 3.1]{erd}). Concerning the restricted Remez inequality, (\ref{665}) holds
for any doubling weight $u$
(see \cite[Th. 7.2]{erd}) and 
 therefore,
$u\in \mathcal{B}(p)\cap\mathcal{R}_{int}(p)$, $0<p<\infty$. Then, by
Corollary~\ref{remez-int}
$\omega_1\ldots\omega_{s-1}u\in\mathcal{R}_{int}(p)$. Thus,
if $0<p<\infty$, the statement of Theorem \ref{tri**} follows from Theorem \ref{Th5.1'} by induction.

Let now $p=\infty$ and $u\in A^*$. Bernstein's inequality (\ref{cl-b}) is proved in \cite[(6.7)]{mas1} and Remez's inequality in \cite[(6.10)]{mas1}.
Therefore, $u\in A^*$ implies $u\in
\mathcal{B}(\infty)\cap\mathcal{R}_{int}(\infty)$. Similarly to the
 case $p<\infty$, Theorem \ref{tri**} immediately follows from
Corollary~\ref{corcor} and Theorem \ref{Th5.1'}.
\end{proof}

\vspace{6mm}

\section{Weighted Bernstein and Markov inequalities for algebraic polynomials
}\label{section-algebraic}
In this section, we deal with weights $\omega$ and $u : [-1,1] \to [0,\infty)$. The weight $u$ is either doubling or satisfies the $A^*$
condition on $[-1,1]$ which are defined similar to those on $\T$
(see, e.g., \cite[p. 62]{mas1}). First, we obtain the weighted
Bernstein inequality for algebraic polynomials on $[-1,1].$

\begin{theorem}\label{t71}
Let $0<p<\infty$, $\omega=\omega_1\ldots\omega_s$, where  $\omega_i(\cos t)\in \Omega$, $i=1,\ldots, s$,  and a weight $u$ is doubling.
Then
\begin{equation}\label{eq71}
\int_{-1}^1 \varphi^p(x) \,|P'_n(x)|^p\, {\omega}(x) u(x) \,dx
\le C(p, \omega, u)\, n^p
\int_{-1}^1 |P_n(x)|^p \,{\omega}(x) u(x) \,dx, \qquad \varphi(x)=\sqrt{1-x^2}.
\end{equation}
\end{theorem}
 \begin{proof}
This result immediately follows from Theorem \ref{tri**}, change of variables $x=\cos t$, and  the fact that
$$
u\quad \mbox{is doubling on}\; [-1,1] \quad \mbox{if and only if}\quad u(\cos t)|\sin t|\quad \mbox{is doubling on}\; \T,
$$
 see \cite[p. 63]{mas1}.
 \end{proof}
A counterpart for $p=\infty$ reads as follows.
\begin{theorem}
Let $\omega=\omega_1\ldots\omega_s$, where  $\omega_i(\cos t)\in \Omega$, $i=1,\ldots, s$,  and $u\in A^*$.
Then
\begin{equation}\label{7.2}
\| \varphi\,P'_n\, {\omega} u\|_{L_\infty[-1,1]}
\le C( \omega, u)\, n
\| P_n\, {\omega} u\|_{L_\infty[-1,1]}.
\end{equation}
\end{theorem}
The proof is similar to the proof of Theorem \ref{t71} using the fact (\ref{A-zvezda}).

Let us now discuss Markov's inequality for algebraic polynomials.
\begin{theorem}\label{ttt222}
Let $0<p<\infty$, $\omega=\omega_1\ldots\omega_s$, where  $\omega_i(\cos t)\in \Omega$, $i=1,\ldots, s$, and a weight $u$ is doubling.
Then
\begin{equation}
\int_{-1}^1  \,|P'_n(x)|^p\, {\omega}(x) u(x) \,dx
\le C(p, \omega, u)\, n^{2p}
\int_{-1}^1 |P_n(x)|^p \,{\omega}(x) u(x) \,dx.
\end{equation}
\end{theorem}
 \begin{proof}
First, applying the Bernstein inequality (\ref{eq71}),
\begin{equation*}
C
 n^{2p}
\int_{-1}^1 |P_n(x)|^p \,{\omega}(x) u(x) \,dx
\ge  n^{p}
 \int_{-1}^1 \varphi^p(x)  \,|P'_n(x)|^p\, {\omega}(x) u(x) \,dx.
\end{equation*}
Therefore, it is enough to show that
\begin{equation*}
C  n^{p}
 \int_{-1}^1 \varphi^p(x)  \,|P'_n(x)|^p\, {\omega}(x) u(x) \,dx\ge
 \int_{-1}^1   \,|P'_n(x)|^p\, {\omega}(x) u(x) \,dx,
\end{equation*}
or, taking  an even trigonometric polynomial $T_n(t)=P'_n(\cos t)$,
\begin{equation*}
C  n^{p}
 \int_{\T}   \,|T_n(t) \sin t|^p\, {\omega}(\cos t) u(\cos t) |\sin t|\,dt
 \ge   \int_{\T}   \,|T_n(t) |^p\, {\omega}(\cos t) u(\cos t) |\sin t|\,dt,
\end{equation*}
or, equivalently,
\begin{equation*}
C
  n^{p}
 \int_{\T}   \,|T_n(t) \sin t|^p\, \bar{\omega}(t) \bar{u}(t) \,dt
 \ge   \int_{\T}   \,|T_n(t) |^p\, \bar{\omega}(t) \bar{u}(t) \,dt,
\end{equation*}
where $\bar{\omega}(t)= \bar{\omega}_1(t)\ldots\bar{\omega}_s(t),$
$\bar{\omega}_i = \omega_i(\cos t) \in \Omega$, $i=1,\ldots, s$, and $\bar{u}(t) = u(\cos t) |\sin t|$ is doubling on $\T$.

Since any doubling weight $\bar{u}$ satisfies $\bar{u}\in \mathcal{R}_{int}(p)$, $0<p<\infty$,
using Corollary \ref{cor-remez-remez}, it remains to obtain
\begin{equation}\label{501}
C  n^{p} \int_{\T}   \,|T_n(t) \sin t|^p\, |v_{Kn}(t)|^p \bar{u}(t) \,dt
 \ge   \int_{\T}   \,|T_n(t) |^p\, |v_{Kn}(t)|^p \bar{u}(t) \,dt,
\end{equation}
where $v_{Kn}=v^{(1)}_{Kn}  \cdots v^{(s)}_{Kn}$ and $v^{(i)}_{n}$
is the $n$-th partial Fourier sum of $\bar{\omega_i}^{1/p}$.

Moreover, by Theorem 3.1 and Lemma 3.2 from the paper \cite{mas1} for $1\le p<\infty$ and Theorem 2.1 from the paper \cite{erd} for $0<p<1$,
it follows that for any doubling weight $\bar{u}$ there exists a nonnegative trigonometric polynomial $\bar{u}_n$
of degree at most $n$ such that
\begin{equation*}
 \int_{\T}   \,|T_n(t)|^p\,  \bar{u}(t) \,dt
\asymp  \int_{\T}   \,|T_n(t) |^p\, \bar{u}^p_n(t) \,dt,\qquad 0<p<\infty.
\end{equation*}
Then (\ref{501}) follows from
\begin{equation*}
C(p)  n^{p} \int_{\T}   \,|T_n(t) \sin t|^p\, dt
 \ge   \int_{\T}   \,|T_n(t) |^p\,dt
\end{equation*}
for any trigonometric polynomial $T_n$. The latter is known for $1\le p<\infty$  (see \cite[Theorem 1]{bar}) and
the proof in the case of $0<p<1$ is similar.
 \end{proof}
Markov's inequality for the case $p=\infty$ is written as follows.
\begin{theorem}
Let $\omega=\omega_1\ldots\omega_s$, where  $\omega_i(\cos t)\in \Omega$, $i=1,\ldots, s$,
 and  $u$ is an $A^*$ weight on $[-1,1]$. Then
\begin{equation*}
\| \,P'_n\, {\omega} u\|_{L_\infty[-1,1]}
\le C( \omega, u)\, n^2
\| P_n\, {\omega} u\|_{L_\infty[-1,1]}.
\end{equation*}
\end{theorem}
The proof repeats the argument of the proof of Theorem \ref{ttt222} using the following inequality:
\begin{equation*}
\|T_n(t) \|_{C(\T)}\le (n+1) \|T_n(t) \sin t\|_{C(\T)};
\end{equation*}
see \cite{bern, bar}.
\begin{remark}\label{remark-lub}
{\textnormal{
Note that for some weights the Bernstein inequality (\ref{7.2}) for algebraic polynomials can be derived from
known results.
First, let us recall the definition of the Mhaskar-Rakhmanov-Saff number, which is
 a crucial concept to analyze weighted inequalities.
Let us suppose that $\omega(x)=\exp(Q(x)),$ where $Q:(-1,1)\to \mathbb{R}$  is even, and differentiable on $(0,1)$. Also suppose
that $xQ'(x)$ is positive and increasing in $(0,1)$ with limits zero and infinity at zero and 1, respectively, and
$$
\int_0^1\frac{xQ'(x)}{\sqrt{1-x^2}}dx=\infty.
$$
Then the $n$-th Mhaskar-Rakhmanov-Saff number, $a_n=a_n(Q)$, is defined to be the root of
$$n=
\frac2\pi\int_0^1 \frac{a_n xQ'(a_n x)}{\sqrt{1-x^2}}dx, \quad n\ge 1.
$$
The importance of this number lies in the Mhaskar-Saff identity
$$
\|P_n\omega\|_{C[-1,1]}
=
\|P_n\omega\|_{C[-a_n,a_n]}, \qquad n\ge 1,
$$
and asymptotically as $n\to\infty$, $a_n$ is the smallest such
number; see \cite{mh1, mh2}. In particular, for the weight
\begin{equation}\label{omega0}
\omega_\alpha(x)=\exp (-(1-x^2)^\alpha, \;\alpha>0
\end{equation}
we have
\begin{equation}\label{omega1}
1-a_n\asymp  n^{-1/(\alpha+\frac12)}, \qquad n\to \infty.
\end{equation}
For this weight,  Lubinsky and Saff proved the following inequalities \cite[p. 531]{lub-saff}:
\begin{equation}\label{lub1}
\Big| \,P'_n(x)\, {\omega_\alpha(x)} \sqrt{1-\frac{|x|}{a_n}}\Big|\le C(\alpha) n 
\| P_n\, {\omega_\alpha} \|_{C[-1,1]}, \qquad |x|<a_n,
\end{equation}
and
\begin{equation}\label{lub2}
\| \,P'_n\, {\omega_\alpha} \|_{C[-1,1]}
\le C(\alpha) n^\frac{2\alpha+2}{2\alpha+1} 
\| P_n\, {\omega_\alpha} \|_{C[-1,1]}.
\end{equation}
In fact, similar results hold for wide class of functions denoted by $\mathcal{W}$. By definition, $\omega=\exp({-Q})\in \mathcal{W}$, if
\begin{itemize}
  \item[(i)] \;\; $Q$ is even and continuously differentiable in $(-1,1)$, while $Q''$ is continuous in $(0,1)$;
  \item[(ii)]\;\;  $Q'\ge 0$ and $Q''\ge 0 $ in $(0,1)$;
  \item[(iii)] \;\; $\int_0^1 xQ'(x)/(\sqrt{1-x^2})\,dx=\infty$;
  \item[(iv)] \;\; for $T(x)=1+\frac{xQ''(x)}{Q'(x)}$, $x\in (0,1)$ one has:
$T$ is increasing in $(0,1)$, $T(0+)>1$, and\\ $T(x)=O(Q'(x))$, $x\in 1-.$
\end{itemize}
Let us show that both (\ref{lub1}) and (\ref{lub2}) imply
(\ref{7.2}) for $\omega_\alpha$ given by (\ref{omega0}) and
$u(x)\equiv 1$. Indeed, let $x\in (0,1)$. If $1-C^2
n^{-1/(\alpha+1/2)}\le x$ for some positive $C=C(\alpha)$, we have
$n^\frac{2\alpha+2}{2\alpha+1}\le 2 C\frac{n}{\sqrt{1-x^2}}$ and
(\ref{lub2}) implies
\begin{equation}\label{lub3}
| \,P'_n(x)\,\varphi(x) {\omega}_\alpha(x) |
\le C n 
\| P_n\, {\omega_\alpha} \|_{C[-1,1]}
\end{equation}
for such $x$.
\\
If $x\le (C^2-1)/(C^2/a_n-1)$, then $\sqrt{1-x^2}\le 2C
\sqrt{1-\frac{|x|}{a_n}}$ and (\ref{lub1}) implies (\ref{lub3}) for
such $x$.
 Further, (\ref{omega1}) yields that $a_n>1-B n^{-1/(\alpha+1/2)}$ for some $B=B(\alpha)>0$.
Then, taking $C^2=2B+2$, we have
 $$1-C^2 n^{-1/(\alpha+1/2)}<(C^2-1)/(C^2/a_n-1)$$ for sufficiently large $n$.
 Finally, we have
$$\|P_n' \varphi \omega_\alpha\|_{C[-1,1]}\le C n\,\|P_n\omega_\alpha\|_{C[-1,1]}.$$
We also mention that in the recent papers \cite{in2,
in1} the authors obtained the weighted Bernstein, Nikolskii, and
Remez inequalities for  algebraic polynomials for the weights
$\omega(x)=\exp{(-(1-x^2)^\alpha)}u(x)$,\, $\alpha>0$, where $u$ is
doubling on $[-1,1]$. }}
\end{remark}

\vspace{6mm}

\section{Weighted Nikolskii inequalities}
\label{section-nikolskii}
Nikolskii's inequality for trigonometric polynomials, that is,
$$
\|T_n\|_{L_q(\T)}  \le \,C\, n^{1/p-1/q}\|T_n\|_{L_p(\T)}, \qquad p<q,
$$
plays an important role in approximation theory and functional analysis, in particular, to prove embedding theorems for function spaces
(see, e.g., \cite{dw}). It
is known that if $u$ is an $A_\infty$ weight, then for any $0<p\le
q<\infty$ there is a constant $C=C(u,p,q)$ such that
\begin{equation} \label{nikolskii}
\left(\int_{\T} |T_n|^q u \right)^{1/q} \le \,C\, n^{1/p-1/q}\left( \int_{\T} |T_n|^p u^{p/q}\right)^{1/p};
\end{equation}
see \cite[Th. 5.5]{mas1} and \cite[Th. 8.1]{er}.
Moreover, if $u\in A^*$, then for any $1\le p<\infty$
there is a constant $C=C(u,p)$ such that
\begin{equation} \label{nikolskii1}
\|T_n u \|_{L_{\infty(\T)}} \le \,C\, n^{1/p}\left( \int_{\T}
|T_n|^p
u^{p}\right)^{1/p}; 
\end{equation}
see \cite[(6.9)]{mas1}. Note that (\ref{nikolskii1}) holds for $0<
p<1$ as well, provided $u\in A^*$. Indeed,  we first apply
(\ref{nikolskii1}) with $p=1$ to get
\begin{equation}\label{f1}
\|T_n\|_{L_\infty(u)} \le \,C\, n \|T_n\|_{L_1(u)}.
\end{equation}
Second, since $u\in A^*$ yields $u\in A_\infty$, we use (\ref{nikolskii}) with $0<p<1$  and $q=1$:
\begin{equation}\label{f2}
 \|T_n\|_{L_1(u)}
 \le \,C\, n^{\frac1p-1} \|T_n\|_{L_p(u^p)}.
\end{equation}
We prove the following weighted Nikolskii inequalities for trigonometric
polynomials.
\begin{theorem}\label{nikol-theorem}
 Let $0<p\le q\le\infty$,
 $\omega=\omega_1\ldots\omega_s$, where  $\omega_i\in \Omega$, $i=1,\ldots, s$, and  $u\in \mathcal{R}_{int}(q)$.
\\
{\bf (A)}.
\ Let $q<\infty$ and $u^{p/q} \in \mathcal{R}_{int}(p)$. Suppose $u$ is such that inequality (\ref{nikolskii})
 holds for each trigonometric polynomial $T_n$. Then
\begin{equation}\label{nikol1}
\|T_n\|_{L_q(\omega u)} \le \,C\, n^{1/p-1/q}
\|T_n\|_{L_p\big((\omega u)^{p/q}\big)}, 
\end{equation}
where $C=C(\omega,u, p,q)$.
\\
{\bf (B)}. 
\ Let $p<q=\infty$ and  $u^{p} \in \mathcal{R}_{int}(p)$. Suppose $u$ is such that inequality (\ref{nikolskii1})
holds for each trigonometric polynomial $T_n$. Then
\begin{equation}\label{nikol2}
\|T_n\|_{L_\infty(\omega u)} \le \,C\, n^{1/p}
\|T_n\|_{L_p\big((\omega u)^{p}\big)},
\end{equation}
where $C=C(\omega,u, p)$.

\end{theorem}
In particular, this implies
\begin{corollary}\label{nikol-corollary}
 Let
  $\omega=\omega_1\ldots\omega_s$, where  $\omega_i\in \Omega$, $i=1,\ldots, s$.
   Then
inequality (\ref{nikol1}) holds provided $u\in A_\infty$ and
$0<p\le q<\infty$
and
(\ref{nikol2}) holds provided $u\in A^*$ and $0< p<\infty$.
\end{corollary}

\begin{proof}[Proof of Theorem \ref{nikol-theorem}]
First, by definition of the class $\Omega$, any weight $\omega_i\in \Omega$, $1\le i\le s-1$, is such that
$\omega_i^{p/q} 
 \in \Omega$ for any $0<p,q<\infty$. Then, by Corollary \ref{remez-int} we get that
  $\big(\omega_1\ldots\omega_{s-1} u\big) \in \mathcal{R}_{int}(q)$ and
  $\big(\omega_1\ldots\omega_{s-1} u\big)^{p/q} \in \mathcal{R}_{int}(p)$.
  Thus, it is enough to prove (\ref{nikol1}) and (\ref{nikol2}) for
$\omega=\omega_s\in \Omega$.
\\
(A). \ \ By (\ref{500}) we have
\begin{equation}\label{nikol-v}
\int_{\T} |T_n|^q \omega u\, \asymp \int_{\T} |T_{n}|^q \,
|v_{Kn}|^q u, \qquad u\in \mathcal{R}_{int}(q),
\end{equation}
where $v_{n}$ is the $n$-th partial Fourier sum of $\omega^{1/q}$
and $K=K(\omega,u)$ is large enough. Moreover, applying again
(\ref{500}) for the weight $\omega^{p/q}$, where $0<p\le q<\infty,$
we have
\begin{equation}\label{nikol-v}
\int_{\T} |T_n|^p \omega^{p/q} u^{p/q}\,
\asymp
\int_{\T} |T_{n}|^p \, |v_{Kn}|^p u^{p/q}
\end{equation}
for $K=K(\omega,u,p,q)$ large enough, provided that $u^{p/q}\in
\mathcal{R}_{int}(p)$. Now we apply (\ref{nikolskii})  to get
(\ref{nikol1}).
\\
(B). \ \  The case $q=\infty$ is similar since
$\big(\omega_1\ldots\omega_{s-1} u\big) \in \mathcal{R}_{int}(\infty)$
and
$\big(\omega_1\ldots\omega_{s-1} u\big)^{p} \in \mathcal{R}_{int}(p)$.
\end{proof}

\begin{proof}[Proof of Corollary \ref{nikol-corollary}]
To show (\ref{nikol1}) for $0<p<q<\infty$ and  (\ref{nikol2}) for
$1\le p<\infty$, we use  results from \cite{mas1}, \cite{er},
and the following two facts:\\
(i)\; $u^{p/q}\in A_\infty$ whenever $u\in A_\infty$ and $0<p<q<\infty$ (see \cite[Ch.
V]{stein}) and
\\
(ii)\; $u^{p}\in A^*\subset A_\infty$ whenever $u\in A^*$ and $p>1$.
The
latter follows from Jensen's inequality.

 To prove (\ref{nikol2}) for $0<p<1$, we first
 apply  (\ref{nikol2}) with $p=1$ and then (\ref{nikol1}) with $0<p<1$  and $q=1$ as in (\ref{f1}) and (\ref{f2}).

\end{proof}

We finish this section by proving Nikolskii's inequalities for algebraic polynomials.
\begin{corollary}
 Let $0<p\le q\le\infty$ and
 $\omega=\omega_1\ldots\omega_s$, where  $\omega_i(\cos t)\in \Omega$, $i=1,\ldots, s$.
 Then for each algebraic polynomial $P_n$ we have
\begin{equation}\label{nikol12}
\|P_n\|_{L_q([-1,1], \omega u)} \le \,C(p,q,\omega,u)\, n^{2/p-2/q}
\|P_n\|_{L_p\big([-1,1],(\omega u)^{p/q}\big)},\qquad\qquad 0<p\le
q<\infty,
\end{equation}
provided $u\in A_\infty$ and
\begin{equation}\label{nikol13}
\|P_n\|_{L_\infty([-1,1],\omega u)} \le \,C(p,q,\omega,u)\, n^{2/p}
\|P_n\|_{L_p\big([-1,1],(\omega u)^{p}\big)},\qquad\qquad
0<p<\infty,
\end{equation}
provided $u\in A^*$.
\end{corollary}

\begin{proof}

First, let $0<p\le q<\infty$. We give a straightforward proof applying the Remez inequalities for algebraic polynomials given by Corollary \ref{corollary-remez}. Define
$$E:=\Big\{x\in [-1,1]: n^2\int_{-1}^1 |P_n|^q \omega u\le |P_n(x)|^q\omega(x)u(x)\Big\}.$$
Then, since $|E|\le n^{-2}$ inequality (\ref{vspom}) yields

\begin{eqnarray*}
\|P_n\|_{L_q([-1,1], \omega u)}^q &\le& C(q,\omega,u)
\int\limits_{[-1,1]\backslash E} |P_n|^q \omega u
\\
&\le& C(q,\omega,u) \Big\| |P_n|^q \omega
u\Big\|^{(q-p)/q}_{L_\infty({[-1,1]\backslash E})}
\int\limits_{[-1,1]\backslash E} |P_n|^p (\omega u)^{p/q}
\\
&\le& \,C(q,\omega,u) n^{2 (q-p)/q} \Big(\int\limits_{-1}^1
|P_n|^q \omega u \Big)^{(q-p)/q} \int\limits_{-1}^1 |P_n|^p
(\omega u)^{p/q},
\end{eqnarray*}
which gives (\ref{nikol12}).


Let now
 $0<p<\infty$ and $u\in A^*$.
Let
 $v^{(i)}_{n}(\cos t)$ be the $n$-th partial Fourier sum of
$\omega_i(\cos t)\in\Omega$, $i=1,\ldots, s$. Then, by
Corollary \ref{cor-remez-remez} changing variables gives
$$
\Big\|
P_n  \omega u
\Big\|_{L_\infty[-1,1]} \asymp
\Big\|
P_n  v^{(1)}_{Kn} \cdots v^{(s)}_{Kn} u \Big\|_{L_\infty[-1,1]},
$$
provided that $u(\cos t)|\sin t|$ is an $A^*$ weight on $\T$.
The latter holds from 
 (\ref{A-zvezda}).

Moreover, since $u^p\in A^*\subset A_\infty$, $p>1$,
 Corollary \ref{cor-remez-remez} implies that
$$
\int_{-1}^1 |P_n |^p (\omega u)^{p} \asymp \int_{-1}^1
|P_n |^p |v^{(1)}_{Kn}|^p \cdots |v^{(s)}_{Kn}|^p u^{p}.
$$
Then (\ref{nikol13}) for $1\le p<\infty$ follows from
$$
\|P_n\|_{L_\infty(u)} \le \,C (p,u)\, n^{2/p}
\|P_n\|_{L_p(u^{p})}, \qquad u\in A^*,\qquad 1\le p<\infty;
$$
see \cite[(7.31)]{mas1}.
The case $0<p<1$ can be treated as in the proof of
Corollary \ref{nikol-corollary}.

\end{proof}

\vspace{6mm}

\section{Necessary conditions for weighted Bernstein inequality}
\label{section-necessary}
We will use the following properties of the Chebyshev polynomials
$\mathcal{T}_{n}$ defined by $\mathcal{T}_{n} (\cos t)= \cos nt$:
\begin{equation}
\label{p1} |\mathcal{T}_{n}(x)|\le 1,\qquad |x|\le 1;
\end{equation}
\begin{equation}
\label{p2} \quad \mathcal{T}_{n}(x)\quad\text{is increasing on }\,
(1,\infty);
\end{equation}
\begin{equation}
\label{p3}
\mathcal{T}_n(x)=\frac{1}{2}\Big(\big(x+\sqrt{x^2-1}\big)^n+\big(x-\sqrt{x^2-1}\big)^n\Big)\qquad
\text{for every}\quad x\in \mathbb{R}\setminus (-1,1).
\end{equation}
The last identity readily implies that
\begin{equation}
\label{p4}\mathcal{T}_n\big(1+\frac 1{n^2}\big)\le C_1,\qquad n\in \N,
\end{equation}
and
\begin{equation}
\label{p5}\frac{\mathcal{T}_n'(x)}{\mathcal{T}_n(x)}\ge \frac14\frac
n{\sqrt{x^2-1}},\qquad x>1+\frac 1{n^2}, \qquad n\in \N.
\end{equation}
To prove the main theorems of this section we need two auxiliary results.

\begin{lemma} \label{seq}
Let $ \xi$ be a negative increasing continuous  function on
$(0,\epsilon)$, for some $\epsilon>0$, and such that $ \xi(0+)=-\infty$ and,
for each $r\in(0,1)$,
\begin{equation}
\label{gg1} \frac{ \xi(rx)}{ \xi(x)}\to\infty\qquad\text{as }\quad x\to 0+.
\end{equation}
Then for each positive sequences $h_n$ such that $h_n\to 0$ as
$n\to\infty$ there exists a positive sequence $\beta_n\to 0$ as
$n\to\infty$ such that,
for each $r\in(0,1)$,
$$
\inf_{x\in(0,h_n)}\frac{ \xi(rx)}{ \xi(x)}\beta_n\to\infty \qquad\text{as
} \quad n\to\infty.
$$
\end{lemma}
\begin{proof}Fix a positive sequence $h_n\to 0$.
By~\eqref{gg1}, there exists an increasing sequence of positive
integers $n(k)$ such that for each $n>n(k)$
\begin{equation}
\label{gg2} \inf_{x\in(0,h_n)}\frac{ \xi((1-1/k)x)}{ \xi(x)}>k^2.
\end{equation}
Put $\beta_n=1/k$ for $n\in [n(k)+1,n(k+1)]$. Fix $r\in (0,1)$.
Consider a positive integer $K$ such that $1-1/K>r$ and $h_n<\epsilon$ for $n>n(K)$.
Applying monotonicity of $ \xi$ and~\eqref{gg2}, we get that
$$
\inf_{x\in(0,h_n)}\frac{ \xi(rx)}{ \xi(x)}\beta_n>\inf_{x\in(0,h_n)}\frac{ \xi((1-1/K)x)}{ \xi(x)}\frac
1K> K,\quad n\in [n(K)+1,n(K+1)].
$$
This establishes the statement of the lemma.
\end{proof}
The proof of the next lemma is a trivial corollary of the mean value theorem.
\begin{lemma}
\label{M} Let $ \xi$ be an increasing continuous  function on
$(0,\epsilon)$, for some $\epsilon>0$, and  such that $ \xi(0+)=-\infty$.
Then, for each $M$ large enough, the equation
$$
 \xi(x)=-Mx
$$
has a unique solution $y(M)\in (0, \epsilon)$, which is continuous in $M$ and
decreasing to $0$ as $M\to\infty$.
\end{lemma}

Now we give the following extension of Theorem \ref{b}.
\begin{theorem}
\label{neg} Let $\omega\in C(\T)$ be an arbitrary weight function
satisfying the following conditions:
\begin{equation}
\label{p6} \omega(t_0)=0,\,\,\text{for some}\,\, t_0\in \T,
\end{equation}
\begin{equation}
\label{p7} \omega\,\text{is increasing on }
(t_0,t_0+\epsilon)\,\,\text {and}\,\, \omega\,\text{is decreasing on
} (t_0-\epsilon,t_0)\,\, \text{for some}\,\, \epsilon>0,
\end{equation}
\begin{equation}
\label{p8} \lim_{t\to
t_0}\frac{\log{\omega(t_0+r(t-t_0))}}{\log{\omega(t)}}=\infty,
\quad\text{for each } r\in (0,1).
\end{equation}

Then for each $0<p\le\infty$ there exists a sequence of
trigonometric polynomials $Q_n$ of degree at most $n$ such that
$$
\lim_{n\to\infty}\frac{\|Q_n'\|_{L_p(\omega)}}{n\|Q_n\|_{L_p(\omega)}}=\infty.
$$
\end{theorem}
\begin{remark}  
{(i)}\ \  Note that if $\omega$ is a continuous nondoubling weight then $\omega(t_0)=0$ for some $t_0\in \T$, i.e., condition~\eqref{p6} holds.
 Without loss of generality we assume  below that $t_0=0$ and $\|\omega\|_{C(\T)}\le 1$.
\\
(ii) \ \
 Condition~\eqref{p7} is assumed to simplify the proof. 
   The principal condition
is~\eqref{p8}, which implies that $\omega$ goes to $0$  fast enough
as $t\to 0$. Condition~\eqref{p8} can be equivalently written as
follows: for each $r\in(0,1)$,
$$
\lim_{t\to 0}\frac{\log{\omega(rt)}}{\log{\omega(t)}}\quad
\text{exists or equal}\;\, \infty,
$$
and, for some $r^*\in(0,1)$,
$$
\lim_{t\to 0}\frac{\log{\omega(r^*t)}}{\log{\omega(t)}}=\infty.
$$
\end{remark}
\begin{example}A typical example of  weights satisfying conditions of Theorem
\ref{neg} is
$$
\omega^*_{\alpha} (t) =\exp \big(- F(g(t))\Big), 
$$
where
$$
F(x)= \exp \big(|x|^{-\alpha}\big),
\qquad \alpha>0,
$$
and $g$ is an analytic function, $g(t): \T\to [-1,1]$, $g(0)=0$.
Although $\omega^*_{\alpha}\in C^{\infty}(\T)$, the result of Theorem \ref{tri**} is not true for this kind of functions.
\end{example}

\begin{proof}[Proof of Theorem~\ref{neg}]

Our proof is in five steps.
 First, we will prove the theorem for $p=\infty$ (steps 1--4).
\\[10pt]
\underline{Step 1.}
 Recall that $t_0=0$ and $\|\omega\|_{C(\T)}\le 1$.
We choose $Q_n$ as follows:
$$
Q_n(t):=\mathcal{T}_n(1+a_n^2-\sin ^2t),
$$
where $a_n\to 0$ is a positive sequence depending on $\omega$ to be
chosen later. For each $n\in\N$, we denote by $b_n$ any point on
$\T$ such that
$$
\|Q_n\omega\|_{C(\T)}=|Q_n(b_n)\omega(b_n)|.
$$
Without loss of generality we may assume that $b_n\in(0,\pi)$. Suppose that the
sequence $\{a_n\}$ is such that
\begin{equation}
\label{p9} \lim_{n\to\infty}Q_n(b_n)\omega(b_n)=\infty,
\end{equation}
and
\begin{equation}
\label{p10} b_n=a_n(1+o(1))\qquad \text{as}\quad n\to\infty.
\end{equation}
Then~\eqref{p4} and~\eqref{p9} imply
\begin{equation} \label{p10.5}
1+a_n^2-\sin^2b_n>1+\frac 1{n^2}
\end{equation}
for $n$ large enough.

Hence,
\begin{eqnarray*}
\frac{\|Q_n'\omega\|_{C(\T)}}{n\|Q_n\omega\|_{C(\T)}}&\ge&
\frac{|Q_n'(b_n)\omega(b_n)|}{nQ_n(b_n)\omega(b_n)}=
\frac{\mathcal{T}'_n(1+a_n^2-\sin^2b_n)|\sin
2b_n|}{n\mathcal{T}_n(1+a_n^2-\sin^2b_n)}\\
&\ge& \frac{|\sin 2b_n|}{4\sqrt{(1+a_n^2-\sin^2b_n)^2-1}},
\end{eqnarray*}
where in the last inequality we used~\eqref{p5}. Finally, taking
into account~\eqref{p10}, we obtain
$$
\lim_{n\to\infty}\frac{\|Q_n'\omega\|_{C(\T)}}{n\|Q_n\omega\|_{C(\T)}}=\infty,
$$
which is the statement of the theorem in the case $p=\infty$.
\\[10pt]
\underline{Step 2.}
Let us now focus on the search of the sequence $a_n$ which satisfies~\eqref{p9} and~\eqref{p10}.
Note that if we take sequences $a_n\to 0$ and $\lambda_n\to
1$ such that
\begin{equation}
\label{p11}
\mathcal{T}_n(1+a_n^2-\sin^2(\lambda_na_n))\omega(\lambda_na_n)\to\infty\qquad\text{as}\quad
n\to\infty,
\end{equation}
and, for each $r\in(0,1)$,
\begin{equation}
\label{p12} \mathcal{T}_n(1+a_n^2)\omega(ra_n)\to 0
\qquad\text{as}\quad n\to\infty,
\end{equation}
then $a_n$ satisfies~\eqref{p9} and~\eqref{p10}. Indeed, condition~\eqref{p11} immediately implies~\eqref{p9},
so~\eqref{p10.5} holds as well, and hence
$$
\limsup_{n\to\infty}\frac{b_n}{a_n}\le 1.
$$
If
$$
\liminf_{n\to\infty}\frac{b_n}{a_n}<r<1,
$$
then, applying~\eqref{p2} and \eqref{p7}, we have
$$
Q_n(b_n)\omega(b_n)\le\mathcal{T}_n(1+a_n^2)\omega(ra_n)
$$
for infinitely many $n\in\N$. This inequality together
with~\eqref{p12} contradicts~\eqref{p9}. So,
$$\liminf_{n\to\infty}\frac{b_n}{a_n}\ge 1\quad\mbox{ and therefore,}
\qquad
\lim_{n\to\infty}\frac{b_n}{a_n}=1,
$$
which is~\eqref{p10}.
\\[10pt]
\underline{Step 3.} Let us set $$ \xi:=\log \omega.$$ Taking logarithm in the
both sides of~\eqref{p11} and~\eqref{p12}, and applying~\eqref{p3}
we get that if $\{a_n\}$ and $\{\lambda_n\}$ satisfy
$$
n\log\left(1+a_n^2-\sin^2(\lambda_na_n)+\sqrt{(1+a_n^2-\sin^2(\lambda_na_n))^2-1}\right)+ \xi(\lambda_na_n)\to\infty\qquad\text{as}\quad
n\to\infty,
$$
and, for each $r\in(0,1)$,
$$
n\log\left(1+a_n^2+\sqrt{(1+a_n^2)^2-1}\right)+ \xi(ra_n)\to -\infty
\qquad\text{as}\quad n\to\infty,
$$ then $\{a_n\}$ and $\{\lambda_n\}$ satisfy~\eqref{p11} and~\eqref{p12} as well. Finally, since
$\log(1+t+\sqrt{(1+t)^2-1})\sim\sqrt{2t}$ as $t\to 0$ it is enough
to choose $a_n\to 0$ and $\lambda_n\to 1-$ such that
\begin{equation}
\label{p13}
n\lambda_na_n\sqrt{1-\lambda_n^2}+ \xi(\lambda_na_n)\to\infty\qquad\text{as}\quad
n\to\infty,
\end{equation}
and, for each $r\in(0,1)$,
\begin{equation}
\label{p14} 2na_n+ \xi(ra_n)\to -\infty \qquad\text{as}\quad
n\to\infty.
\end{equation}
\\[10pt]
\underline{Step 4.} Now we are in a position to choose $\{a_n\}$ and
$\{\lambda_n\}$. For $n$ large enough, let $h_n$ be a unique
solution of the equation
$$
 \xi(x)=-n^{1/2}x,
$$
provided by Lemma~\ref{M}. It follows from Lemma~\ref{seq} that there exists a
sequence $\{\lambda_n\}$ which goes to $1$ slow enough such that
$$\sqrt{1-\lambda_n^2}>n^{-1/3}
$$ and, for each $r\in(0,1)$,
$$
\inf_{t\in(0,2h_n)}\frac{ \xi(rt)}{ \xi(t)}\sqrt{1-\lambda_n^2}\to\infty
\qquad\text{as }\quad n\to\infty.
$$
Moreover, for each $r\in(0,1)$ and $r_1\in (r,1)$, we have
\begin{eqnarray}
\nonumber
 \inf_{t\in(0,h_n/\lambda_n)}\frac{ \xi(rt)}{ \xi(\lambda_nt)}\sqrt{1-\lambda_n^2}
 &=&\inf_{t\in(0,h_n)}\frac{ \xi(rt/\lambda_n)}{ \xi(t)}\sqrt{1-\lambda_n^2}
 \\
 \nonumber
 &\ge&
 \inf_{t\in(0,h_n/\lambda_n)}\frac{ \xi(rt/\lambda_n)}{ \xi(t)}\sqrt{1-\lambda_n^2}
 \\
 \label{p15}
 &\ge&\inf_{t\in(0,h_n/\lambda_n)}\frac{ \xi(r_1t)}{ \xi(t)}\sqrt{1-\lambda_n^2}\to\infty
\quad\text{as }n\to\infty.
\end{eqnarray}
Put $a_n:=z_n/\lambda_n$, where $z_n$ is a unique solution of the
equation
\begin{equation}
\label{p144}
 \xi(z)=-\frac 12nz\sqrt{1-\lambda_n^2},
\end{equation}
provided by Lemma~\ref{M}. Then, Lemma~\ref{M} implies that  $z_n\to
0$, and hence $a_n\to 0$. Therefore,
$$
n\lambda_na_n\sqrt{1-\lambda_n^2}+ \xi(\lambda_na_n)=- \xi(\lambda_na_n)\to
\infty \qquad\text{as }\quad n\to\infty,
$$
i.e., \eqref{p13} holds.

On the other hand, Lemma~\ref{M} together with the condition $\frac
12n\sqrt{1-\lambda_n^2}>n^{1/2}$ for $n$ large enough implies that
$z_n=a_n\lambda_n< h_n$. Thus,~\eqref{p15} yields
$$
\lim_{n\to\infty}\frac{ \xi(r a_n)}{ \xi(\lambda_n
a_n)}\sqrt{1-\lambda_n^2}=\infty.
$$
Moreover, (\ref{p144}) implies
$$
 \xi(\lambda_na_n)=-\frac{1}{2}n\lambda_na_n\sqrt{1-\lambda_n^2}.
$$
Hence,
$$
\lim_{n\to\infty}\frac{ \xi(r a_n)}{{n}\lambda_n a_n/2}=-\infty,
$$
which gives \eqref{p14}.

Thus, the sequence $\{a_n\}$ satisfies ~\eqref{p13} and~\eqref{p14}
and therefore, \eqref{p9} and \eqref{p10}, which concludes the proof
of Theorem~\ref{neg} in the case $p=\infty$.
\\[10pt]
\underline{Step 5.}
 The proof for the case $0<p<\infty$ follows the same lines as the one for the case $p=\infty.$
  We again choose the polynomial $Q_n$ as
$$
Q_n(t):=\mathcal{T}_n(1+a_n^2-\sin ^2t),
$$
where $a_n=a_n(\omega, p)\to 0$ is a positive sequence 
 to be chosen later.  Similarly to  Steps 1 and  2, it is enough to
find a sequence $\{a_n\}$ such that $a_n\to 0$ as $n\to\infty$, for each $r\in(0,1)$,
$$
|\mathcal{T}_n(1+a_n^2)|^p\omega(ra_n)\to 0\quad\text{as}\,\,n\to\infty,
$$
 and
$$
\int_{\T}|\mathcal{T}_n(1+a_n^2-\sin
^2t)|^p\omega(t)dt\to\infty\quad\text{as}\,\,n\to\infty.
$$
The latter holds if for some sequence $\{\lambda_n\}$ such that
$\lambda_n\to 1-$ one has
\begin{align*}
&\int_{(2\lambda_n-1)a_n}^{\lambda_n a_n}
|\mathcal{T}_n(1+a_n^2-\sin^2t)|^p\omega(t)dt
\\
&\ge
\qquad
|\mathcal{T}_n(1+a_n^2-\sin
^2(\lambda_na_n))|^p\,\omega\big((2\lambda_n-1)a_n\big)(1-\lambda_n)a_n\to\infty.
\end{align*}
Similarly to the Step 3 (cf. ~\eqref{p13} and~\eqref{p14}) it is enough to choose sequences
$\{\lambda_n\}$ and $\{a_n\}$ such that
\begin{equation}
\label{p131} pn\lambda_na_n\sqrt{1-\lambda_n^2}+\log{(1-\lambda_n)}
+\log{a_n} + \xi\big((2\lambda_n-1)a_n\big)\to\infty\qquad\text{as}\quad
n\to\infty,
\end{equation}
and for each $r\in(0,1)$
\begin{equation}
\label{p141} 2pna_n+ \xi(ra_n)\to -\infty \qquad\text{as}\quad
n\to\infty.
\end{equation}
Similarly to the Step 4 one can choose sequences $\{\lambda_n\}$ and
$\{a_n\}$ satisfying
\begin{equation}
\label{p132}  \xi\big((2\lambda_n-1)a_n\big)=-pn\lambda_na_n(1-\lambda_n^2)
\end{equation}
and~\eqref{p141}. Finally, note that~\eqref{p132} together with
$\lim_{n\to\infty}\omega(a_n)/a_n=0$ implies~\eqref{p131}.
\end{proof}
The next theorem (cf. Theorem 1.2) is the main negative result in
the paper providing a necessary condition for the weighted
Bernstein inequality to hold.

\begin{theorem}
\label{neg11} Let $\omega\in C(\T)$ be an arbitrary weight
function satisfying (\ref{p6}), (\ref{p7}), and the following condition:
\begin{equation}\label{limitlimit}
 \limsup_{t\to t_0}\frac{\log{\omega(t_0+r(t-t_0))}}{\log{\omega(t)}}=\infty
\qquad\text{for each } \, r\in (0,1).
\end{equation}
Then for each $0<p\le\infty$ there exists a sequence of positive
integers $K_n\to\infty$ as $n\to\infty$, and a sequence of
trigonometric polynomials $Q_n$ of degree at most $K_n$ such that
$$
\lim_{n\to\infty}\frac{\|Q_n'\|_{L_p(\omega)}}{K_n\|Q_n
\|_{L_p(\omega)}}=\infty.
$$
\end{theorem}
\begin{remark}
If condition (\ref{limitlimit}) holds for some $r\in (0,1)$, then it also holds for any $r\in (0,1)$.
\end{remark}

\begin{proof}
Without loss of generality we assume  below that $t_0=0$ and
$\|\omega\|_{C(\T)}\le 1$. We will prove the theorem only for the
case $p=\infty$. The case $0<p<\infty$ is similar (see the proof of
Theorem~\ref{neg}, Step 5). Define $Q_n$ as follows:
$$
Q_n(t):=\mathcal{T}_{K_n}(1+a_n^2-\sin ^2t),
$$
where $K_n$ and $a_n\to 0$ to be chosen later. Put $ \xi:=\log \omega.$
Now proceeding step by step the proof of Theorem~\ref{neg} up to
~\eqref{p13} and~\eqref{p14} one can see that it is enough to choose
$a_n\to 0$, an increasing sequence of integers $\{K_n\}$, and
$\lambda_n\to 1-$ such that
\begin{equation}
\label{p113}
K_n\lambda_na_n\sqrt{1-\lambda_n^2}+ \xi(\lambda_na_n)\to\infty\qquad\text{as}\quad
n\to\infty
\end{equation}
and, for each $r\in(0,1)$,
\begin{equation}
\label{p114} 2K_na_n+ \xi(ra_n)\to -\infty \qquad\text{as}\quad
n\to\infty.
\end{equation}
Since
$$
\limsup_{t\to 0}\frac{ \xi(rt)}{ \xi(t)}=\infty \qquad\text{for each }
r\in (0,1),
$$
there exists decreasing positive sequence $c_n$ such that
$c_n\to 0$ as $n\to\infty$, and
\begin{equation}
\label{ww1} \frac{ \xi((1-1/n)c_n)}{ \xi(c_n)}>n^2.
\end{equation}
Put
$$
\lambda_n:=1-1/n,\quad a_n:=c_n/\lambda_n,\quad \text{ and}\quad
K_n:=2\left[\frac{- \xi(c_n)}{\lambda_na_n\sqrt{1-\lambda_n^2}}\right].
$$
Since $\lim_{t\to 0} \xi(t)=-\infty$, then $K_n\to\infty$ as
$n\to\infty$, and hence~\eqref{p113} holds.

To complete the proof, take an
arbitrary $r\in(0,1)$. Since $r<\lambda_n^2$ for $n$ large enough,
by monotonicity of $ \xi$,
$$
2K_na_n+ \xi(ra_n)<2K_na_n+ \xi((1-1/n)c_n).
$$
Thus, by~\eqref{ww1}
$$
2K_na_n+ \xi(ra_n)<2K_na_n+n^2 \xi(c_n)\le\frac{-4 \xi(c_n)}{\lambda_n\sqrt{1-\lambda_n^2}}+n^2 \xi(c_n)\to-\infty
$$
as $n\to\infty$. This proves~\eqref{p114}.
\end{proof}
The next theorem shows an essential difference between
Theorems~\ref{neg} and ~\ref{neg11} in the case when the weight
satisfies (\ref{limitlimit}) but not (\ref{p8}). In this case
Bernstein's inequality may hold for some subsequence of integers
$\{K_n\}$ but not for all $n\in\N$. For simplicity we consider only
the case $p=\infty$ and $t_0=0$.
\begin{theorem}
\label{neg2}

There exists an even weight function $\omega\in C^{\infty}(\T)$
satisfying~\eqref{p6} and~\eqref{p7} and
\begin{equation}
\label{e20} \limsup_{t\to 0}\frac{\log \omega(rt)}{\log
\omega(t)}=\infty, \quad\text{for each } r\in (0,1),
\end{equation}
such that for some increasing sequence of positive integers $K_n$
the Bernstein inequality
$$
\|T_n'\omega\|_{C(\T)}\le CK_n\|T_n\omega\|_{C(\T)}
$$
holds for any trigonometric polynomial $T_n$ of degree at most
$K_n$.
\end{theorem}
\begin{proof}
Let
$$
W(x)=\frac{\int^{\pi x}_0\exp\big(-1/\sin^2t\big)dt}{\int^{\pi
}_0\exp\big(-1/\sin^2t\big)dt}, \quad x\in [0,1].
$$

Define an even weight $\omega$ as follows
\begin{equation*}
\omega(t):=\begin{cases}
1, &\text{if \qquad $t\in[\alpha_1,\pi]$,}\\&\\
d_n,&\text{if \qquad $t\in[\alpha_n,\frac{\alpha_{n-1}}2]$, $n\ge
2$,}\\&\\
d_{n+1}+(d_n-d_{n+1})W\left(\frac{2t}{\alpha_n}-1\right), &\text{if
\qquad $t\in[\frac{\alpha_{n}}2,\alpha_n]$, $n\ge
1$,}\\&\\
0, &\text{if \qquad $t=0$,}
\end{cases}
\end{equation*}
where $d_n:=\exp\big(-\exp (n^2)\big)$ and $\alpha_n:=d_n^2$. By construction,
$\omega\in C^{\infty}(\T)$. Since
$$
\lim_{n\to\infty}\frac{\log \omega(\alpha_n/2)}{\log
\omega(\alpha_n)}=\infty,
$$
then $\omega$ satisfies~\eqref{e20}.

 For each $n\in \N$, we also define
an even weight $\omega_n$:
\begin{equation*}
\omega_n(t):=\begin{cases}
\omega(t), &\text{if $ t\in[\alpha_n,\pi]$,}\\&\\
d_n,& \text{if $t\in[0,\alpha_n]$}.
\end{cases}
\end{equation*}
Put
\begin{equation}
\label{e1} K_n:=\left[\frac{1}{100\alpha_n}\right].
\end{equation}
Take a polynomial $T_n$ of degree at most $K_n$. Since
$\omega_n(t)\ge \omega(t)$, $t\in\T$, then
$\|T_n\omega\|_{C(\T)}\le\|T_n\omega_n\|_{C(\T)}$.

On the other hand,
\begin{equation}
\label{e2} \|T_n\omega_n\|_{C(\T)}\le 2\|T_n\omega\|_{C(\T)}.
\end{equation}
Indeed, let $t_0\in\T$ be a point where $|T_n\omega_n|$ attains its maximum.
If $|t_0|\ge\alpha_n$, then~\eqref{e2} is obvious. If
$|t_0|<\alpha_n$, then using Remez's inequality and~\eqref{e1} we get
\begin{eqnarray}\nonumber
\|T_n\omega_n\|_{C(\T)}&=&d_n\|T_n\|_{C(\T)}\le
d_n \exp({8\alpha_nK_n})\max_{t\in\T\setminus[-\alpha_n,\alpha_n]}|T_n(t)|
\\
\label{e0}
&<&2\max_{t\in\T\setminus[-\alpha_n,\alpha_n]}|T_n(t)\omega(t)|\le 2
\|T_n\omega\|_{C(\T)}.
\end{eqnarray}

Note that by definition of $w_n$ we have
\begin{equation}
\label{e4} |\omega'_n(t)|\le C\max_{1\le k\le
n-1}\frac{d_k}{\alpha_k}\le C\frac{d_{n-1}}{\alpha_{n-1}}, \quad
t\in \T.
\end{equation}
Hence,
\begin{equation}
\label{e3} |\omega'_n(t)|\le C\max_{1\le k\le
n-1}\frac{d_k}{d_{k+1}\alpha_k}|\omega_n(t)|\le
C\frac{d_{n-1}}{d_n\alpha_{n-1}}|\omega_n(t)|, \quad t\in \T.
\end{equation}
Moreover, we have
\begin{equation}
\label{e5} |\omega''_n(t)|\le C\max_{1\le k\le
n-1}\frac{d_k}{\alpha^2_k}\le C\frac{d_{n-1}}{\alpha^2_{n-1}}, \quad
t\in \T.
\end{equation}

Since $\omega_n\in C^{\infty}(\T)$ then by Jackson's theorem there
exists a trigonometric polynomial $Q_n$ of degree $K_n$ such that
$$
\|\omega_n-Q_n\|\le C\frac{\|\omega'_n\|_{C(\T)}}{K_n}
$$
and
$$
\|\omega'_n-Q'_n\|\le C\frac{\|\omega''_n\|_{C(\T)}}{K_n}.
$$
Thus, \eqref{e4} and~\eqref{e5} yield that
\begin{equation}
\label{e6} \|\omega_n-Q_n\|\le
C\frac{d_{n-1}\alpha_n}{\alpha_{n-1}}\le \frac{d_n}2
\end{equation}
and
\begin{equation}
\label{e7} \|\omega'_n-Q'_n\|\le
C\frac{d_{n-1}\alpha_n}{\alpha^2_{n-1}}\le K_nd_n
\end{equation}
for $n$ large enough. Now by~\eqref{e6} we get
\begin{eqnarray*}
\|T'_n\omega\|_{C(\T)}&\le& \|T'_n\omega_n\|_{C(\T)}\le
\|T'_nQ_n\|_{C(\T)}+\|T'_n\|_{C(\T)}\|\omega_n-Q_n\|_{C(\T)}
\\
&\le&\|T'_nQ_n\|_{C(\T)}+\frac{d_n}{2}\|T'_n\|_{C(\T)}\le\|T'_nQ_n\|_{C(\T)}+
\frac 12\|T'_n\omega_n\|_{C(\T)}.
\end{eqnarray*}
Therefore,
$$
\|T'_n\omega\|_{C(\T)}\le\|T'_n\omega_n\|_{C(\T)}\le
2\|T'_nQ_n\|_{C(\T)}.
$$
Similarly applying the inequality
$$
\|T'_n\omega_n\|_{C(\T)}\ge\|T'_nQ_n\|_{C(\T)} - \|T'_n\|_{C(\T)} \|\omega_n-Q_n\|_{C(\T)},$$
we get
\begin{equation}
\label{e66} \|T_nQ_n\|_{C(\T)}\le 2\|T_n\omega_n\|_{C(\T)}.
\end{equation}

Thus,
$$
\|T'_n\omega\|_{C(\T)}\le 2\|T'_nQ_n\|_{C(\T)}\le
2\|(T_nQ_n)'\|_{C(\T)}+2\|T_nQ'_n\|_{C(\T)}=:I_1+I_2.
$$
By Bernstein's inequality for the polynomials and~\eqref{e6} we have
$$
I_1\le CK_n\|T_nQ_n\|_{C(\T)}\le 4CK_n\|T_n\omega\|_{C(\T)}.
$$
Regarding  $I_2$, we first note that
$$
I_2\le 2\|T_n\omega_n'\|_{C(\T)}+
2\|T_n\|_{C(\T)}\|\omega'_n-Q'_n\|_{C(\T)}=:I_{21}+I_{22}.
$$
By~\eqref{e3} and~\eqref{e0} we get
$$
I_{21}\le
C\frac{d_{n-1}}{d_n\alpha_{n-1}}\|T_n\omega_n\|_{C(\T)}<CK_n\|T_n\omega\|_{C(\T)}.
$$
Moreover, \eqref{e7} and \eqref{e0} imply
$$
I_{22}\le 2K_nd_n \|T_n\|_{C(\T)}\le 2 K_n\|T_n\omega_n\|_{C(\T)}\le
4K_n\|T_n\omega\|_{C(\T)}
$$ for $n$ large enough.
Hence, for any $n\in \N$,
$$
\|T_n'\omega\|_{C(\T)}\le CK_n\|T_n\omega\|_{C(\T)}.
$$
\end{proof}

{\bf Acknowledgements.}
The authors  are grateful to D. Lubinsky and V. Totik for valuable discussions.
 The authors thank the Centre de Recerca Matem\`{a}tica, the Centre for Advanced Study at the Norwegian Academy of Science and Letters in Oslo, and Mathematisches Forschungsinstitut Oberwolfach for their hospitality during the
preparation of this manuscript and for providing a stimulating
atmosphere for research.

\end{document}